\documentclass[11pt,A4]{article}
\usepackage{amsmath,amsthm,amsfonts,amssymb,amscd, amsxtra,color}
\usepackage[active]{srcltx}
\usepackage{url}
\usepackage{cite}
\usepackage{multirow}
\usepackage{colortbl,cancel}
\usepackage{graphicx}
\usepackage{placeins}
\usepackage{float}
\usepackage{subfigure}
\usepackage[margin=2.5cm,nohead]{geometry}
\newtheorem{theorem}{Theorem}
\newtheorem{lemma}{Lemma}
\newtheorem{corollary}{Corollary}

\newtheorem{remark}{Remark}
\newtheorem{definition}{Definition}
\newtheorem{example}{Example}

\usepackage{amssymb,amsfonts,amstext,amsmath}
\usepackage{xcolor}
\usepackage{cite}

\DeclareMathOperator{\grad}{grad}  
\usepackage{amssymb}

\newtheorem{algorithm}{Algorithm}

\newtheorem{strategy}{Strategy}

\usepackage{geometry}
\usepackage{color}
\usepackage{float}
\usepackage{textcomp}
\usepackage{amsthm}
\usepackage{amsmath}
\usepackage{amssymb}

 \usepackage{colortbl,hhline,color,soul,url}

\usepackage{changes}

\usepackage{mathtools}

\begin{document}

\title{Gradient Method  for Optimization on Riemannian Manifolds with Lower Bounded Curvature} 

\author{
O. P. Ferreira \thanks{IME/UFG, Avenida Esperan\c{c}a, s/n, Campus Samambaia,  Goi\^ania, GO, 74690-900, Brazil (e-mails: {\tt orizon@ufg.br},  {\tt mauriciosilvalouzeiro@gmail.com},  {\tt lfprudente@ufg.br}).}
\and
 M. S. Louzeiro \footnotemark[1]
\and
L. F. Prudente \footnotemark[1]
}

\maketitle
\begin{abstract}

The gradient  method for minimize a differentiable convex function on Riemannian manifolds  with lower bounded sectional curvature is  analyzed in this paper.  The analysis of  the method is presented  with three different finite procedures for determining the stepsize, namely, Lipschitz  stepsize,  adaptive stepsize  and  Armijo's stepsize.   The first procedure requires that the objective function has Lipschitz continuous gradient, which is not necessary for the other approaches. Convergence of the whole sequence to a minimizer,  without any level set boundedness assumption,  is  proved.   Iteration-complexity bound for  functions with  Lipschitz continuous gradient is also presented. Numerical experiments are provided to illustrate the effectiveness of the method in this new setting  and certify the obtained theoretical results. In particular, we consider the problem of finding the Riemannian center of mass and the so-called Karcher's mean. Our numerical experiences indicate that the adaptive stepsize is a promising scheme that is worth considering.

\noindent{\bf Keywords:}  Gradient method, convex programming,   Riemannian manifold,  lower bounded curvature, iteration-complexity bound.\\
\noindent{\bf AMS subject classification:} \,90C33\,$\cdot$\,49K05\,$\cdot$\, 47J25
\end{abstract}
\section{Introduction}

We consider the   gradient method to solve the optimization problem defined by:
\begin{equation} \label{eq:OptP}
\min \{ f(p) ~:~   p\in \mathcal{M}\},
\end{equation}
where the constraint set $\mathcal{M}$ is  endowed  with a structure of a  {\it complete   Riemannian manifold with lower bounded curvature}  and $f:\mathcal{M}\to \mathbb{R}$ is a {\it continuously differentiable convex function}.  It is well known that, in several cases,    by endowing $\mathcal{M}$ with a suitable  Riemannian metric, an Euclidean non-convex constrained problem  can be seen as a Riemannian convex unconstrained  problem.  In addition to this property, we will  present some examples showing  that  endowing the set of constraints with a suitable  Riemannian metric the objective function can be also {\it Riemannian  Lipschitz gradient}.   Consequently,   the geometric and algebraic structure that comes from of the Riemannian metric makes  possible to greatly reduce the computational cost for solving such  problems.  Indeed, it is also widely known that,  in several contexts,  the iteration complexity of the gradient method for convex optimization problems with Lipschitz gradient  is much lower than  for general nonconvex problems; see for example~\cite{BentoFerreiraMelo2017,JeurisVandebrilVandereycken2012,Rapcsak1997,SraHosseini2015, ZhangReddiSra2016} and references therein.  Furthermore, many Euclidean optimization problems are naturally posed on the Riemannian context;  see \cite{EdelmanAriasSmith1999,JeurisVandebrilVandereycken2012,  Smith1994, SraHosseini2015}. Then,  to take advantage of the Riemannian geometric  structure, it is preferable to treat these  problems as the ones of  finding singularities of gradient vector fields  on  Riemannian manifolds rather than using Lagrange multipliers  or projection methods; see \cite{Luenberger1972,   Smith1994, Udriste1994}.   Accordingly, constrained optimization problems can be viewed as unconstrained ones from a Riemannian geometry point of view. Moreover, Riemannian structures can also opens up new research directions  that aid in developing competitive algorithms;  see \cite{AbsilMahonySepulchre2008, EdelmanAriasSmith1999, JeurisVandebrilVandereycken2012, NesterovTodd2002, Smith1994, SraHosseini2015}.   For this purpose,  extensions of concepts and techniques of optimization from Euclidean  space to Riemannian context have been quite frequently in recent years. Papers dealing with this subject include, but are not limited to \cite{LiMordukhovichWang2011,   LiYao2012, WangLiWangYao2015,WangLiYao2015,Manton2015,ZhangReddiSra2016, ZhangSra2016}.

The gradient method is one of the oldest methods for the minimization of a differentiable function in Euclidean space.  Despite having  slow convergence rate, the simplicity of implementation, the low memory requirements and cost per iteration,  make the gradient  method quite attractive  to solve large-scale optimization problems. Indeed, the computational cost per iteration  is mildly dependent on the dimension of the problem, yielding computational efficiency for  this method; see\cite{JeurisVandebrilVandereycken2012, Nesterov2013, Raydan1997}.   In addition, the gradient method is the starting point for  designing many more sophisticated and efficient algorithms, including   fast gradient method,  accelerated gradient method and Barzilai-Borwein method; see \cite{Nesterov2004, Yuan2008} for a  comprehensive study on this subject.    To the best of our knowledge  the gradient method was the first optimization method to be considered in a Riemannian setting.  In order to deal with contained optimization problems in the Euclidean space,  Luenberger~\cite{Luenberger1972}   proposed and established  important  convergence properties of gradient method by using    the  Riemannian  structure of the constraint set induced  by  the  Euclidean  structure.  Since then, the gradient method  has been studied in  general  Riemannian  manifold. Some early works dealing with this method  include \cite{ Gabay1982, Udriste1994, Smith1994, Rapcsak1997}.   However,   the obtained  convergence results  in these previous works demand that the initial points  of the sequence  belong to a bounded level set of the objective function  establishing  only that all its cluster points are stationary.   By assuming  convexity of the objective function and that the  manifolds has {\it non-negative  curvature}, it has been proven  in \cite{daCruzNetoLimaOliveira1998} that,  for a suitable choice of the stepsize and without any level set boundedness assumption,  the whole sequence converges to a solution.  Recently  new  important properties of the gradient method in Riemannian settings  have been obtained. For instance,  in \cite{ZhangSra2016}  the authors  provided iteration-complexity bounds for convex optimization problems on Hadamard manifolds.  In \cite{BoumalAbsilCartis2016}, the authors established iteration-complexity bounds   without any assumption on the    convexity of the  problem and curvature of the manifold. In \cite{BiniIannazzo2013}  the   gradient method is considered  to compute  the  Karcher mean, which is a strong convex function in the cone of symmetric positive definite matrices  endowed with a suitable  Riemannian metric.  In  \cite{AfsariTronVidal2013} is  studied  properties  of the gradient method  for the problem of finding the global Riemannian center of mass of a set of data points on a Riemannian manifold. In \cite{glaydstontr} is extended the convergence analysis  of the gradient method to the Hadamard setting for continuously differentiable functions which satisfy the Kurdyka-Lojasiewicz inequality.

By the aforementioned we see that the gradient method remains a subject of considerable interest.  In spite of  its long history,  the full convergence of the sequence generated by the gradient  method in a general   Riemannian manifolds  has not yet been established.  However, as far as we know,  the full convergence of the sequence generated by the gradient  method under convexity of the objective function and {\it lower boundedness  of the curvature} of  Riemannian manifolds  is a new contribution of this paper, which adds important results in the available convergence theory of this  method. The analysis of  the method is presented  with three different finite procedures for determining the stepsize, namely, Lipschitz  stepsize,  adaptive stepsize  and  Armijo's stepsize. It should be noted that  we use a recent inequality established in\cite{WangLiWangYao2015, WangLiYao2015}. Numerical experiments are provided to illustrate the effectiveness of the method in this new setting  and certify the  obtained theoretical results.  In particular, we consider  the problem of finding the  Riemannian mass center and the so-called Karcher's mean. Our experiments indicate that adaptive size is a promising scheme that is worth considering.

This paper is organized as follows.  Section~\ref{sec:aux} presents some definitions and  preliminary results  related to the Riemannian geometry  that are important throughout our study.  In Section~\ref{sec:Gradient}, we state the gradient algorithm and the three different finite procedures for determining the  stepsize. Section~\ref{sec.main}  is devoted to the asymptotic convergence analysis of the method, and  in Section~\ref{Sec:IteCompAnalysis} the  iteration-complexity bound is presented.  Section~\ref{Sec:Examples}   provides some  examples of functions satisfying the assumptions of our  results in the previous sections. In Section~\ref{Sec:NumExp},  we present some numerical experiments to illustrate the behavior of the method. The last section contains some  conclusions.
\section{Notations and basic concepts} \label{sec:aux}
In this section, we recall  some  concepts, notations, and basics results  about Riemannian manifolds.   For more details we refer the reader to \cite{doCarmo1992, Sakai1996, Udriste1994,  Rapcsak1997}.

We denote by $T_p\mathcal{M}$ the {\it tangent space} of a finite dimensional  Riemannian manifold $\mathcal{M}$ at $p$.  The corresponding norm associated to the Riemannian metric $\langle \cdot ,  \cdot \rangle$ is denoted by $\|  \cdot \|$.   We use $\ell(\alpha)$ to denote the length of a piecewise smooth curve $\alpha:[a,b]\to \mathcal{M}$. The Riemannian  distance  between $p$ and $q$   in  $\mathcal{M}$ is denoted  by $d(p,q)$,  which induces the original topology on $\mathcal{M}$, namely,  $( \mathcal{M}, d)$, which is a complete metric space where bounded and closed subsets are compact. The {\it closed metric ball} in $\mathcal{M}$ centered at the point $p\in\mathcal{M}$ with radius $r>0$ is denoted by $B[p,r]$. Denote by ${\cal X}(\mathcal{M})$, the space of smooth vector fields on $\mathcal{M}$. Let $\nabla$ be the Levi-Civita connection associated to $(\mathcal{M}, \langle \cdot,\cdot \rangle)$. For each $t \in [a,b]$ and a piecewise smooth curve $\alpha:[a,b]\to \mathcal{M}$,  $\nabla$ induces an isometry  relative to $ \langle \cdot , \cdot \rangle  $,  $P_{\alpha,a,t} \colon T _{\alpha(a)} {\mathcal{M}} \to T _
{\alpha(t)} {\mathcal{M}}$  defined by $ P_{\alpha,a,t}\, v = V(t)$, where $V$ is the unique vector field on the curve $\alpha$ such that $ \nabla_{\alpha'(t)}V(t) = 0$ and $V(a)=v$.  The isometry $P_{\alpha,a,t}$  is called {\it parallel transport} along  of  $\alpha$ joining  $\alpha(a)$ to $\alpha(t)$ and,  when there is no confusion,  it will be denoted by  $P_{\alpha,p,q}$.  A vector field $V$ along a smooth curve $\gamma$  is said to be {\it parallel} iff $\nabla_{\gamma^{\prime}} V=0$. If $\gamma^{\prime}$ itself is parallel, we say that $\gamma$ is a {\it geodesic}. Given that the geodesic equation $\nabla_{\ \gamma^{\prime}} \gamma^{\prime}=0$ is a second order nonlinear ordinary differential equation, then the geodesic $\gamma=\gamma _{v}( \cdot ,p)$ is determined by its position $p$ and velocity $v$ at $p$. It is easy to check that $\|\gamma ^{\prime}\|$ is constant. The restriction of a geodesic to a  closed bounded interval is called a {\it geodesic segment}.  A geodesic segment joining $p$ to $q$ in $\mathcal{M}$ is said to be {\it minimal} if its length is equal to $d(p,q)$. A Riemannian manifold is {\it complete} if the geodesics are defined for any values of $t\in \mathbb{R}$. Hopf-Rinow's theorem asserts that any pair of points in a  complete Riemannian  manifold $\mathcal{M}$ can be joined by a (not necessarily unique) minimal geodesic segment.  Owing to  the completeness of the Riemannian manifold $\mathcal{M}$, the {\it exponential map} $\exp_{p}:T_{p}  \mathcal{M} \to \mathcal{M} $ is  given by $\exp_{p}v\,=\, \gamma _{v}(1,p)$, for each $p\in \mathcal{M}$.   {\it In this paper, all manifolds are assumed to be  connected,   finite dimensional, and complete}.  For $f: {\cal D} \to\mathbb{R}$ a   differentiable function on the open  set  ${\cal D} \subset \mathcal{M}$,   the Riemannian metric induces the mapping   $f\mapsto  \mbox{grad} f $  associates     its {\it gradient} via the following  rule  $\langle \mbox{grad} f(p),X(p)\rangle\coloneqq d f(p)X$, for all $p\in {\cal D}$. For a twice-differentiable function,  the mapping  $f\mapsto \mbox{hess} f$ associates    its  {\it hessian} via the rule 
$ \langle \mbox{hess} f X,X\rangle \coloneqq d^2 f(X, X)$,  for all $X \in {\cal X}({\cal D})$, where    the last equalities imply that $ \mbox{hess} f X= \nabla_{X}  \mbox{grad} f$,  for all $X \in {\cal X}({\cal D})$.  We proceeded to recall  some concepts and basic properties  about  convexity   in the  Riemannin context.   For more details   see, for example, \cite{Udriste1994,  Rapcsak1997, WangLiWangYao2015}.  For any two points $p,q\in\mathcal{M}$, $\Gamma_{pq}$ denotes the set of all geodesic segments  $\gamma:[0,1]\rightarrow\mathcal{M}$ with $\gamma(0)=p$ and $\gamma(1)=q$.   We use $\Gamma^{\Omega}_{pq}$  to denote the set of all $\gamma\in\Gamma_{pq}$ such that $\gamma(t)\in \Omega$, for all $t\in [0,1]$.   A nonempty subset  $\Omega \subset \mathcal{M}$ is said to be {\it weakly convex} if, for any $p,q\in \Omega$, there is a minimal geodesic segment joining   $p$ to $q$  belonging  $\Omega$.  A   function $f: \mathcal{D} \to {\mathbb{R}}$  is said to be {\it convex} on the set $\Omega\subset \mathcal{D}$ if $\Omega$ is weakly convex and for any $p,q\in\Omega$ and $\gamma\in\Gamma^{\Omega}_{pq}$  the composition $f\circ\gamma:[0, 1]\to\mathbb{R}$ is a convex function on $[0,1]$, i.e.,   $f\circ\gamma(t)\leq(1-t)f(p)+tf(q)$, for all $t\in[0,1]$; see \cite{WangLiWangYao2015}. For $f$  a differentiable function on $\mathcal{D}$ and  a weakly convex set  $\Omega \subset \mathcal{D}$,  we have the following  characterization:   $f$ is convex  on $\Omega$ iff there holds $f(\gamma(t))\geq f(p)+\langle \grad f(p), \gamma'(0)\rangle$, for all $p,q\in\Omega$ and $\gamma\in\Gamma^{\Omega}_{pq}$.

The following lemma plays  an important role in next sections and its proof, with some minor technical adjustments,  can be found in \cite[Lemma~3.2]{WangLiWangYao2015}; see also \cite{WangLiYao2015}. For simplifying  the notations,  let 
\begin{equation} \label{eq:KappaHat}
\kappa<0, \qquad \hat{\kappa}:=\sqrt{|\kappa|}.
\end{equation}
\begin{lemma} \label{lemli}
Let $\mathcal{M}$ be a  Riemannian manifolds with sectional curvature $K\geq\kappa$,  and $\hat{\kappa}$ be defined in \eqref{eq:KappaHat}. Assume that  $f$ is  differentiable and convex  on the set $\Omega\subset \mathcal{M}$, $p\in \Omega$  and ${\gamma}:[0,\infty)\to \mathcal{M}$ is   defined by $ {\gamma}(t)=\mbox{exp} _{p}\left(-t\ \grad f(p)\right).$
Then, for any $t\in[0,\infty)$ and $q\in \Omega$ there holds
\begin{multline*}
\cosh(\hat{\kappa}d(\gamma(t),q))\leq \cosh(\hat{\kappa}d(p,q))+\\
\hat{\kappa}\cosh(\hat{\kappa}d(p,q))\sinh(t\hat{\kappa}\left\|\grad f(p)\right\|)\left[\frac{t\left\|\grad f(p)\right\|}{2}-\frac{\tanh(\hat{\kappa}d(p,q))}{\hat{\kappa}d(p,q)} \frac{f(p)-f(q)}{\left\|\grad f(p)\right\|}\right]
\end{multline*}
and, consequently, the following inequality  holds 
\begin{multline*}
d^2({\gamma}(t),q)\leq d^2(p,q) + \\
 \frac{\sinh\left(\hat{\kappa}t\|\grad f(p)\|\right)}{\hat{\kappa}} \left[t\|\grad f(p)\|\,\frac{\hat{\kappa}d(p,q)}{\tanh\left(\hat{\kappa}d(p,q)\right)}-\frac{2}{\left\|\grad f(p)\right\|}\left(f(p)-f(q)\right)\right]. 
\end{multline*}
\end{lemma}
Next we present the definition of Lipschitz continuous gradient vector field; see \cite{da1998geodesic}.
\begin{definition} \label{Def:GradLips}
Let $f$ be a differentiable function on the set ${\cal D}$. The gradient vector  field    of $f$  is  said  to be   Lipschitz continuous on ${\cal D}$ with   constant $L\geq 0$ if,   for any $p,q\in{\cal D}$ and $\gamma\in\Gamma^{{\cal D}}_{pq}$, it holds that 
$\left\|P_{{\gamma},p,q} \grad f(p)- \grad f(q)\right\|\leq L\ell(\gamma).$
\end{definition}
 The {\it norm of the hessian}   $\mbox{hess}\,f$   at $p\in{\mathcal M}$  is given by
\begin{equation} \label{eq:DefNormHess}   
\|\mbox{hess}\,f(p)\|:= \sup  \left\{ \left\|\mbox{hess}\,f(p)v\right \|~:~  v\in T_{p}\mathcal{M}, ~\|v\|=1\right\}.
\end{equation} 
 In the following result we present a characterization for  twice continuously differentiable  functions with Lipschitz continuous gradient vector  field, which has similar  proof of its Euclidean counterpart   and will be omitted here.
\begin{lemma} \label{le:CharactGL}
Let $f:\mathcal{D} \to \mathbb{R}$ be a twice continuously differentiable  function. The gradient vector  field   of $f$  is Lipschitz continuous with   constant $L\geq 0$ if,  and only if,  there exists $L\geq 0$ such that $\|\mbox{hess}\,f(p)\|\le L$, for all $p\in \mathcal{D}$.
\end{lemma}

The next  lemma  can be found in  \cite[Corollary 2.1]{BentoFerreiraMelo2017} with minor adjustment. Its  proof follows from  the definition  of convexity of functions and the fundamental theorem of calculus.
\begin{lemma} \label{le:lc}
Let $f$ be a differentiable function on the set ${\cal D}$ and $a>0$. Assume that  $\grad f$  is   Lipschitz continuous   on ${\cal D}$ with   constant $L\geq 0$ and $p\in \Omega$.  If  $ \exp_{p}(- t \grad f(p))\in {\cal D}$, for all $ t\in [0,a]$, then   there holds
\[
f( \exp_{p}(- t \grad f(p))) \leq f(p)- \left(1-\frac{L}{2}t\right)t\left\|\grad f(p)\right\|^{2}, \qquad  \forall~t\in [0,a].
\]
\end{lemma} 
Note that  if ${\cal D}= {\cal M}$,  then  condition  $\exp_{p}(- t \grad f(p))\in {\cal D}$, for all $ t\in [0,a]$, in Lemma~\ref{le:lc} plays no role. In the following example we present  a functions satisfying all the assumptions of  Lemma~\ref{le:lc} for the case    ${\cal D}\neq {\cal M}$.
\begin{example}\label{ex:pdist}
Let $\mathcal{M}=\{ p\in \mathbb{R}^n:~\|p\|=1\}$ the Euclidean sphere and $q\in \mathcal{M}$. Define  $\varphi_q(p):=d^2(p,q)/2$, for all $p\in \mathcal{M}$. The function  $\varphi_q$ is  differentiable in  ${\cal D}:=\{p\in\mathcal{M}:d(p,q)< 5\pi/6 \}$ and convex in $\Omega:=\{p\in\mathcal{M}:d(p,q)\leq\pi/2 \}$. Furthermore,     $\grad \varphi_q$  is   Lipschitz continuous   on ${\cal D}$, because  ${\cal D}$ is compact and  $\mbox{hess}~ \varphi_q$ is continuous in $\mathcal{M}\backslash\{-q\}\supset{\cal D}$. Indeed,  combining \cite[Lemma 3]{FerreiraIusemNemeth2014} with  Lemma~\ref{le:CharactGL} we conclude that 
$$
L= \sup_{p\in{\cal D}} \frac{\left|\left\langle p,q\right\rangle\arccos \left\langle p,q\right\rangle\right|}{\sqrt{1-\left\langle p,q\right\rangle^2}}= \frac{5 \pi}{6} \sqrt{3}.
$$
 Since $\grad \varphi_q(p)=-\exp^{-1}_{p}q$  for all $p\in \mathcal{M}\backslash\{-q\}$, after some calculations, we conclude that $d(\exp_{p}(- t \grad \varphi_q(p)), p)\leq td(p,q)$, for all $p\in {\cal D}$. Hence, letting $p\in \Omega$ we have
$$
d((\exp_{p}(- t \grad \varphi_q(p)),q)\leq d(\exp_{p}(- t \grad \varphi_q(p)),p)+d(p,q)\leq(t+1)\frac{\pi}{2}, 
$$
and then $ \exp_{p}(- t \grad \varphi_q(p))\in {\cal D}$, for all $ t\in [0, 1/L]$. 
For more details about the  function $\varphi_q$;   see \cite{FerreiraIusemNemeth2014}.
\end{example}
The following concept will be useful in the analysis of the sequence generated by the gradient method. In fact, as we shall prove,  the sequence generated by this method satisfies the following definition. 
\begin{definition} \label{def:QuasiFejer}
A sequence $\{y_k\}$ in the complete metric space $(\mathcal{M},d)$ is quasi-Fej\'er convergent to a set $W\subset \mathcal{M}$ if, for every $w\in W$, there exist  a sequence $\{\epsilon_k\}\subset\mathbb{R}$ such that $\epsilon_k\geq 0$, $\sum_{k=1}^{\infty}\epsilon_k<+\infty$, and $d^2(y_{k+1},w)\leq d^2(y_k,w)+\epsilon_k$, for all $k=0, 1, \ldots$.
\end{definition}
The main property of a quasi-Fej\'er  sequence is stated  in the next result, and its proof is similar to the one proved in    \cite{burachik1995full},  by replacing the Euclidean distance by the Riemannian. 
\begin{theorem}\label{teo.qf}
Let $\{y_k\}$ be a sequence in the complete metric space $(\mathcal{M},d)$. If $\{y_k\}$ is quasi-Fej\'er convergent to a nonempty set $W\subset \mathcal{M}$, then $\{y_k\}$ is bounded. If furthermore, a cluster point $\bar{y}$ of $\{y_k\}$ belongs to $W$, then $\lim_{k\to\infty}y_k=\bar{y}$.
\end{theorem}
The study of the  gradient method for convex functions  is well understood for Riemannian manifold with nonnegative sectional curvature and Hadamard manifolds; see \cite{da1998geodesic, ZhangReddiSra2016, ZhangSra2016}. In order to  increase the domain of applications of the method,  {\it hereafter, we assume that $\mathcal{M}$  is a   complete Riemannian manifolds with sectional curvature $K\geq\kappa$,  where $\kappa<0$}, unless the contrary is explicitly stated.  
\section{The Riemannian gradient method}\label{sec:Gradient} 
In this section we state the    Riemannian gradient method   to solve \eqref{eq:OptP} and the strategies for  choosing  the stepsize that will be used in our analysis. 

Let  $f:  \mathcal{D} \to\mathbb{R}$  be   differentiable, $\mathcal{D}\subset \mathcal{M}$ be an open set, $\Omega^*$  be the {\it solution set} of   the problem \eqref{eq:OptP}, $f^*\coloneqq \inf_{x\in \mathcal{D}}f(x)$   be the  {\it optimum value}  of $f$,  and $c\in   \mathbb{R}$. {\it From now on, we assume that  $\Omega^*$ is non-empty  and  $f$   is convex on the sub-level set  ${\cal L}_{c}f$}, where 
$$
{\cal L}_{c}f:=\{p\in  \mathcal{M}:~ f(p)\leq c\} \subset \mathcal{D}.
$$
The  statement  of    {\it Riemannian  gradient algorithm} to solve the problem \eqref{eq:OptP} is as follows.\\

\hrule
\begin{algorithm} \label{sec:gradient}
{\bf  Gradient algorithm in a Riemanian manifold  $\mathcal{M}$\\}
\hrule
\begin{description}
\item[ Step 0.] Let $ p_0\in{\cal L}_{c}f$. Set $k=0$.
\item[ Step 1.] If $\mbox{grad} f(p_k)=0$, then {\bf stop}; otherwise,  choose a stepsize $t_k>0$ and compute
\begin{equation} \label{eq:GradMethod}
p_{k+1}\coloneqq \mbox{exp} _{p_{k}}\left(-t_{k}\,\mbox{grad} f(p_{k})\right).
\end{equation}
\item[ Step 2.]  Set $k\leftarrow k+1$ and proceed to  \textbf{Step~1}.
\end{description}
\hrule
\end{algorithm}
In the following we present three  different strategies for  choosing  the stepsize $t_k>0$ in Algorithm~\ref{sec:gradient}.  In the first  strategy we assume that  $\grad f$  is Lipschitz continuous.
\begin{strategy}[Lipschitz stepsize]\label{fixed.step} 
Assume that  $\grad f$  is   Lipschitz continuous   on $\mathcal{D}$ with   constant $L\geq 0$  and  that $ \exp_{p}(- t \grad f(p))\in \mathcal{D}$, for all $p\in {\cal L}_{c}f$ and $ t\in [0,1/L]$. Let $\varepsilon>0$ and take
\begin{equation}\label{eq:fixed.step}
\varepsilon< t_k\leq  ~ \frac{1}{L}.
\end{equation}
\end{strategy}
\begin{remark}
If ${\cal D}= {\cal M}$,  then condition  $\exp_{p}(- t \grad f(p))\in {\cal D}$, for all $ t\in [0,a]$, in Strategy~\ref{fixed.step}   plays no role.  Recall that the  function in  Example~\ref{ex:pdist}  satisfies this condition  for     ${\cal D}\neq {\cal M}$.
\end{remark}
Despite knowing that  $\grad f$ is Lipschitz continuous, in  general,  the Lipschitz  constant is not  computable. Next  strategy can be used to compute the stepsize   without any Lipschitz condition.  However, as we shall  show,  if  $\grad f$ is Lipschitz  with constant $L>0$ the   stepsize  computed is an approximation to the  stepsize  $1/L$; see  \cite{BeckTeboulle2009}. 
\begin{strategy}[adaptive stepsize]\label{adaptive.step}
Take $\beta\in(0,1)$, $L_0>0$, and $\eta>1$. Set   $t_k\coloneqq L_{k}^{-1} $, where   $L_{k}\coloneqq \eta^{i_k}L_{k-1}$ and 
\begin{equation} \label{eq:GradMethodcpl}
i_k\coloneqq \min \left\{ i~:~f({\gamma_k}(\tau_i))\leq f(p_k)-  \beta \tau_i \left\|\grad f(p_k)\right\|^2,~ i=0,1, \ldots\right\}, 
\end{equation}
where $\tau_i:=(\eta^{i}L_{k-1})^{-1}$  and  ${\gamma_k}(\tau_i)\coloneqq \mbox{exp} _{p_k}\left(-\tau_i\ \grad f(p_k)\right)$.
\end{strategy}

\begin{strategy}[Armijo's stepsize]\label{armijo.step}
Choose $\beta\in(0,1)$ and   take
\begin{equation}\label{step.arm}
t_k\coloneqq \max \left\{ 2^{-i}: ~f\left(\gamma_k (2^{-i})\right)\leq f(p_k)-\beta 2^{-i} \left\|\grad f(p_k)\right\|^2,~ i=0,1, \ldots\right\}, 
\end{equation}
where ${\gamma_k}(2^{-i})\coloneqq \mbox{exp} _{p_k}\left(-2^{-i}\ \grad f(p_k)\right)$.
\end{strategy}
\begin{remark} \label{rem:strategy3}
Strategy~\emph{\ref{adaptive.step}} can be seen as an Armijo-type line search where the first trial stepsize at iteration $k$ is set to be equal to $t_{k-1}$. Indeed, taking $L_0=1$, and $\eta=2$ the inequality in \eqref{eq:GradMethodcpl} can be equivalently rewritten as
$$
f({\gamma_k}( 2^{-i}t_{k-1}))\leq f(p_k)- \beta 2^{-i}t_{k-1} \left\|\grad f(p_k)\right\|^2.
$$
\end{remark}
The proof of the  well-definedness  of  Strategies~\ref{adaptive.step} and ~\ref{armijo.step}    follows the usual arguments  and  will be omitted. On the other hand, \eqref{eq:fixed.step} and  Lemma~\ref{le:lc} imply that, for each $p\in{\cal L}_{c}f$ there holds $ \exp_{p}(- t \grad f(p))\in {\cal L}_{c}f$, for all $ t\in [0,1/L]$. Hence, the sequence $\{p_k\}$ generated by Algorithm~\ref{sec:gradient} with  Strategies~\ref{fixed.step}, ~\ref{adaptive.step} or ~\ref{armijo.step} is well-defined.  Finally we remark that, due to  $f$  be  convex,  $\grad f(p)=0$ if only if   $p\in \Omega^*$.  Therefore, {\it from now on   we assume that  $\grad f(p_k)\neq0$, or equivalently,    $p_k\notin\Omega^*$, for  all $k=0,1,\ldots$.}
\subsection{Asymptotic convergence Analysis } \label{sec.main}
 In this section our goal is to prove that the sequence $\{p_k\}$,   generated by the gradient method  with  Strategies~\ref{fixed.step}, \ref{adaptive.step} or  \ref{armijo.step},   converges to a solution of problem \eqref{eq:OptP}. 
\begin{lemma}\label{lem:boud}
Let $\{p_k\}$ be    generated by  Algorithm~\ref{sec:gradient}  with Strategies~\ref{fixed.step}, \ref{adaptive.step} or  \ref{armijo.step}. Then, 
	\begin{equation}\label{desi.des}
	f(p_{k+1})\leq f(p_k)- \nu  t_k\left\|\grad f(p_k)\right\|^2, \qquad k=0,1, \ldots, 
	\end{equation}
where $\nu=1/2$ for Strategy~\emph{\ref{fixed.step}},  and $\nu=\beta$ for Strategies~\emph{\ref{adaptive.step}} and ~\emph{\ref{armijo.step}}. 	Consequently, 	$\{f(p_k)\}$ is non-increasing sequence  and $\lim_{k\to +\infty}t_k\|\grad f(p_k)\|^2=0$.
\end{lemma}
\begin{proof}
 For  Strategies~\mbox{\ref{adaptive.step}} and {\ref{armijo.step}}, inequality \eqref{desi.des}   follows directly from  \eqref{eq:GradMethodcpl} and \eqref{step.arm}, respectively. Now, we assume that  $\{p_k\}$ is   generated by using Strategy~{\ref{fixed.step}}. In  this case,  Lemma~\ref{le:lc} implies  that 
$$
f(p_{k+1})=f(\mbox{exp} _{p_{k}}\left(-t_{k}\,\mbox{grad} f(p_{k})\right))\leq f(p_k)-\left(1-\frac{{L}}{2}t_k\right)t_k\left\|\grad f(p_k)\right\|^2,  
$$
for all $ k=0,1, \ldots$. Hence, taking into account  \eqref{eq:fixed.step} we  have $1/2\leq (1-Lt_k/2)$ and then,\eqref{desi.des} follows.  Therefore,  \eqref{desi.des} holds for $\{p_k\}$   generated by using the three  strategies. It  is immediate   from \eqref{desi.des} that $\{f(p_k)\}$ is non-increasing. Moreover,  \eqref{desi.des} implies that
$$
\sum_{k=0}^{\ell}t_k\left\|\grad f(p_k)\right\|^2\leq \frac{1}{\nu} \sum_{k=0}^{\ell}f(p_k)-f(p_{k+1})\leq \frac{1}{\nu} \left(f(p_0)-f^*\right), 
$$
for each  nonnegative integer $\ell$,  which implies that $t_k\left\|\grad f(p_k)\right\|^2$ goes to zero,  as $k$ goes to infinity,  completing the proof. 
\end{proof}
\begin{remark} \label{re:apxLipConst}
Whenever  $\grad f$  is   Lipschitz continuous   on $\mathcal{D}$ with   constant $L\geq 0$, the stepsize in Strategy~\ref{adaptive.step} can be seen as  an approximation for the Lipschitz constant. Indeed,  since $L_0>0$ and  $\eta>1$ in Strategy~\ref{adaptive.step}, we conclude that $t_k\coloneqq L_{k}^{-1}\leq L_{k-1}^{-1}=t_{k-1}$, for all $k=0,1, \ldots$. Thus  $t_k\leq 1/L_0$, for all $k=0,1, \ldots$.  If $ L_0\geq L$,  then it follows from  \eqref{desi.des} that $t_k\leq  1/L_0$, for all $k=0,1, \ldots$.  Now assume  that  $L_0\leq L$. In this case,  \eqref{desi.des}  holds for $t_k=1/L$ and then  \eqref{eq:GradMethodcpl} implies  that
$1/(\eta L)\leq t_{k}$.   Therefore,    
\begin{equation}\label{des.arm.aux}
\frac{1}{\eta L}\leq t_k\leq \frac{1}{L_0}, \qquad k=0,1, \ldots
\end{equation}
\end{remark}
Let $ p_0\in \mathcal{M}$.  By Lemma~\ref{lem:boud}, we define   constant $ \rho >0$ as follows
\begin{equation} \label{eq:rho}
\sum_{k=0}^{\infty}t^2_k\left\|\grad f(p_k)\right\|^2\leq \rho :=
\begin{cases}
2[f(p_0)-f^*]/L,      \qquad  \qquad \mbox{ for  Strategy~\mbox{\ref{fixed.step}}};\\
[f(p_0)-f^*]/(\beta L_0),   \qquad \quad  \mbox{for Strategy~\mbox{\ref{adaptive.step}}};\\
[f(p_0)-f^*]/\beta, \quad   \qquad \qquad \mbox{for  Strategy~\mbox{\ref{armijo.step}}}.
\end{cases}
\end{equation}
In the following result, in particular,  we bound the sequence $\{p_k\}$ generated by  Algorithm~\ref{sec:gradient} with Strategies~\ref{fixed.step}, ~\ref{adaptive.step} or ~\ref{armijo.step}.
\begin{lemma}\label{lem:bounded}
Let $q\in\Omega^*$ and $\{p_k\}$ the sequence  generated by  Algorithm~\ref{sec:gradient} with Strategies~\ref{fixed.step}, \ref{adaptive.step} or  \ref{armijo.step}. Then there holds 
\begin{equation}\label{def:dqkapa}
d(p_{k+1},q)\leq \frac{1}{\sqrt{\kappa}}\cosh^{-1}\left(\cosh(\sqrt{\kappa}d(p_0,q))e^{\frac{1}{2}\left(\sqrt{\kappa\rho}\right)\sinh\left(\sqrt{\kappa\rho}\right)}\right), \qquad  k=0, 1, \ldots.
\end{equation}
\end{lemma}
\begin{proof}
Applying  the first inequality of Lemma~\ref{lemli},  with $t=t_k$ and   $p=p_k$ , we have $p_{k+1}=\gamma(t_k)$,  and  taking into account that  $q\in\Omega^*$,   we  conclude that   
$$
\cosh(\hat{\kappa}d(p_{k+1},q))\leq \cosh(\hat{\kappa}d(p_k,q))\left[1+\left( \hat{\kappa} t_k \left\|\grad f(p_k)\right\|\right)^2\frac{\sinh( \hat{\kappa} t_k\left\|\grad f(p_k)\right\|)}{2 \hat{\kappa}t_k\left\|\grad f(p_k)\right\|}\right], 
$$
for all $k=0, 1, \ldots$, where $\hat{\kappa}$ is defined in \eqref{eq:KappaHat}. Since  \eqref{eq:rho} implies  $ t_k \left\|\grad f(p_k)\right\|\leq \sqrt{\rho}$, for all $k=0, 1, \ldots$,  and   the map  $ (0, +\infty) \ni t \mapsto  \sinh(t)/t$ is increasing, we conclude that 
$$\cosh(\hat{\kappa}d(p_{k+1},q))\leq \cosh(\hat{\kappa}d(p_k,q))\left[1+a\left( t_k \left\|\grad f(p_k)\right\|\right)^2\right], \qquad k=0, 1, \ldots,$$
where  $a:= \hat{\kappa}(\sinh(\hat{\kappa}\sqrt{\rho}))/(2\sqrt{\rho})$.  Now note  that  the last inequality implies that 
$$\cosh(\hat{\kappa}d(p_{k+1},q))\leq \cosh(\hat{\kappa}d(p_k,q))e^{a\left( t_k \left\|\grad f(p_k)\right\|\right)^2}, \qquad k=0, 1, \ldots,$$
Therefore, by   using \eqref{eq:rho},  it follows that  $\cosh (\hat{\kappa}d(p_{k+1},q))\leq\cosh(\hat{\kappa}d(p_0,q))e^{a\rho}$, which is equivalent to \eqref{def:dqkapa} by considering the definition of $\hat{\kappa}$ in \eqref{eq:KappaHat}.
\end{proof}
Let us define   the following auxiliary constant   
\begin{equation} \label{eq:qkappa}
{\cal C}_{\rho,\kappa}^q  := \frac{\sinh\left(\sqrt{\kappa\rho}\right)}{\sqrt{\kappa\rho}} \left[1+\cosh^{-1}\left(\cosh(\sqrt{\kappa}d(p_0,q))e^{\frac{1}{2}\left(\sqrt{\kappa\rho}\right)\sinh\left(\sqrt{\kappa\rho}\right)}\right)\right],
\end{equation}
where $\rho$ in  defined in  \eqref{eq:rho}.
\begin{lemma}\label{pr:ltd}
 Let $\{p_k\}$ be    generated by by  Algorithm~\ref{sec:gradient} with Strategies~\ref{fixed.step}, \ref{adaptive.step} or  \ref{armijo.step}. Then,  for each  $q\in\Omega^*$,  there holds
\begin{equation}\label{eq;desgen}
d^2(p_{k+1},q)\leq d^2(p_k,q) +  \frac{t_k}{\nu}{\cal C}_{\rho,\kappa}^q\left[f(p_k)-f(p_{k+1})\right]+2t_k\left[f(q)-f(p_k)\right], 
\end{equation}
for all  $k=0,1,\ldots$, where $\nu=1/2$ for Strategy~\emph{\ref{fixed.step}} and $\nu=\beta$ for Strategies~\emph{\ref{adaptive.step}} and ~\emph{\ref{armijo.step}}.
\end{lemma}
\begin{proof}
Define $\gamma _{k} (t)=\mbox{exp} _{p_{k}}\left(-t\,\grad f(p_k)\right)$,  for all $t\in [0, +\infty)$. Then,   $\gamma_{k}(0)=p_k$ and, from \eqref{eq:GradMethod}, we obtain  $\gamma_{k}(t_k)=p_{k+1}$. Applying second inequality of  Lemma~\ref{lemli} with $\gamma=\gamma_k$, after some manipulations,  we conclude  that 
\begin{multline} \label{eq:GenIneq}
d^2(p_{k+1},q)\leq d^2(p_k,q) +\\
 \frac{\sinh\left(\hat{\kappa}t_k\|\grad f(p_k)\|\right)}{\hat{\kappa}t_k\|\grad f(p_k)\|}\left[t_k^2\|\grad f(p_k)\|^2\,\frac{\hat{\kappa}d(p_k,q)}{\tanh\left(\hat{\kappa}d(p_k,q)\right)}+
2t_k\left[f(q)-f(p_k)\right]\right], 
\end{multline}
for all $k=0,1,\ldots$.   On the other hand,     $t/\tanh(t)\leq 1+t$, for all $t\geq 0$, and the map $ (0, +\infty) \ni t \mapsto  \sinh(t)/t$ is increasing and bounded below by $1$. Thus, taking into account  that  \eqref{eq:rho} implies $ t_k \left\|\grad f(p_k)\right\|\leq \sqrt{\rho}$ for all $k=0, 1, \ldots$, and considering   $f(q)-f(p_k)\leq 0$ for all $k=0,1, \ldots$, we conclude from  \eqref{eq:GenIneq} that 
$$
d^2(p_{k+1},q)\leq d^2(p_k,q) + \frac{\sinh\left(\hat{\kappa}\sqrt{\rho}\right)}{\hat{\kappa}\sqrt{\rho}} t^2_k\|\grad f(p_k)\|^2\left[1+ \hat{\kappa} d(p_k,q)\right]+ 2t_k\,[f(q)-f(p_k)], 
$$
for all $k=0,1,\ldots$, where $\rho$ is defined in \eqref{eq:rho}.  Thus,   by Lemma~\ref{lem:boud},   we obtain  
 $$
d^2(p_{k+1},q)\leq d^2(p_k,q) + \frac{t_k}{\nu}\frac{\sinh\left(\hat{\kappa}\sqrt{\rho}\right)}{\hat{\kappa}\sqrt{\rho}}\left[1+ \hat{\kappa} d(p_k,q)\right][f(p_k)-f(p_{k+1})]+ 2t_k\,[f(q)-f(p_k)], 
$$
 for all $k=0,1,\ldots$. Therefore,  by Lemma~\ref{lem:bounded} and \eqref{eq:qkappa},  we have   \eqref{eq;desgen},  which concludes the proof.
\end{proof}
Finally we are ready to prove the full convergence of  $\{p_k\}$ to a minimizer of $f$.
\begin{theorem}\label{teo.main}
 Let $\{p_k\}$ be    generated by  by  Algorithm~\ref{sec:gradient} with Strategies~\ref{fixed.step}, \ref{adaptive.step} or  \ref{armijo.step}.  Then $\{p_k\}$ converges  to a solution of the problem in \eqref{eq:OptP}. 
\end{theorem}
\begin{proof}
First note that  $f(q)-f(p_k)\leq 0$, for all $k=0,1, \ldots$ and  $q\in\Omega^*$.  Hence,   \eqref{eq:fixed.step},  \eqref{step.arm} and \eqref{des.arm.aux} imply $0<t_k\leq 1/L$ or    $0<t_k\leq 1/L_0$ or $0<t_k\leq 1$, for all $k=0,1, \ldots$, for Strategies~\ref{fixed.step}, \ref{adaptive.step} or  \ref{armijo.step}, respectively. Let  $\Gamma\coloneqq \max\{1, 1/L, 1/L_0\}$. Thus, for  Strategies~\ref{fixed.step}, \ref{adaptive.step} or  \ref{armijo.step} we conclude  from Lemma~\ref{pr:ltd}  that 
$$
d^2(p_{k+1},q)\leq d^2(p_k,q) +\frac{1}{\nu} \Gamma {\cal C}_{\rho,\kappa}^q\left[f(p_k)-f(p_{k+1})\right], \qquad k=0,1,\ldots, 
$$
for all  $q\in\Omega^*$. Considering that   $\sum_{i=0}^\infty [f(p_k)-f(p_{k+1})]\leq [f(p_0)-f^*] $, we conclude that  $\{p_k\}$ is quasi-Fej\'er convergent to $\Omega^*$. Therefore,  since  $\Omega^*$ is non-empty the sequence $\{p_k\}$ is bounded.  Let $\bar{p}$ be an cluster  point of $\left\{p_k\right\}$ and  $\left\{p_{k_j}\right\}$ be  a subsequence $\left\{p_{k}\right\}$ such that  $\lim_{j\to\infty} p_{k_j}= \bar{p}$.   It follows from  Lemma~\ref{lem:boud}  that  $\lim_{k\to\infty}t_k\left\|\grad f(p_k)\right\|^2=0$, and due to  $\{t_k\}$ has a cluster  point $\bar{t}\in [0,\Gamma]$, we analyze the following two possibilities
$$
{\bf (a)}~ \bar{t}>0,  \qquad \qquad   \qquad \qquad    {\bf (b)} ~ \bar{t}=0.
$$
 Assume that~${\bf (a)}$ holds. In this case,  considering that  $\lim_{k\to\infty}t_k\left\|\grad f(p_k)\right\|^2=0$  and $\grad f$ is continuous,    we conclude that 
$$
0=\lim_{j\to\infty} t_{k_j}\left\|\grad f(p_{k_j})\right\|= \bar{t}\left\|\grad f(\bar{p})\right\|.
$$
Hence,   $\grad f(\bar{p})=0$ and then   $\bar{p}\in\Omega^*$. Note that if   Strategy~\mbox{\ref{fixed.step}}   is used, then $\bar{t}$  satisfies  only~{\bf (a)}. Now, we assume that  ${\bf (b)}$ holds.
In this case  Strategies~ ~\mbox{\ref{adaptive.step}}  or \mbox{\ref{armijo.step}}  is used. First assume  Algorithm~\ref{sec:gradient} with Strategy~\mbox{\ref{adaptive.step}}. Since $\{t_{k_j}\}$ converges to $\bar{t}=0$ and $\{t_{k}\}$ is non-increasing, it  follows that $\{t_{k}\}$ converges to $\bar{t}=0$. Hence, taking $r\in\mathbb{N}$  we can conclude that  $t_{k}<(\eta^rL_0)^{-1}$  for $k$  sufficiently   large. Considering $k$ being the smallest natural number that satisfies $t_{k}<(\eta^rL_0)^{-1}$, by   \eqref{eq:GradMethodcpl}, we have
$$
f(\exp_{p_{k}}((\eta^{r-1}L_0)^{-1}[-\grad f(p_{k_j})]))> f(p_{k})-(\eta^{r-1}L_0)^{-1}\beta \left\|\grad f(p_{k})\right\|^2.
$$
Letting $k$ go to $+\infty$ in the above inequality and taking into account that $\grad f$ and the exponential mapping  are  continuous, we obtain
$$
f(\exp_{\bar{p}}((\eta^{r-1}L_0)^{-1}[-\grad f(\bar{p})]))\geq f(\bar{p})-(\eta^{r-1}L_0)^{-1} \beta\left\|\grad f(\bar{p})\right\|^2.
$$
The last inequality is equivalent to
$$
-\frac{f(\exp_{\bar{p}}((\eta^{r-1}L_0)^{-1} [-\grad f(\bar{p})]))-f(\bar{p})}{(\eta^{r-1}L_0)^{-1}}\leq \beta\left\|\grad f(\bar{p})\right\|^2.
$$
Thus, letting $r$ goes  to $+\infty$ we obtain $\left\|\grad f(\bar{p})\right\|^2\leq \beta\left\|\grad f(\bar{p})\right\|^2$ which implies $\grad f(\bar{p})=0$, i.e., $\bar{p}\in\Omega^*$. Therefore, since $\{p_k\}$ is quasi-Fej\'er convergent to $\Omega^*$, we conclude from Theorem~\ref{teo.qf}   that $\{p_k\}$ converges  to $\bar{p}$. Finally,  assume that  Strategy~\mbox{\ref{armijo.step}} is used. Since $\{t_{k_j}\}$ converges to $\bar{t}=0$, taking $r\in\mathbb{N}$,   we  conclude that $t_{k_j}<2^{-r}$ for $j$  sufficiently   large. Thus  Armijo's  condition \eqref{step.arm} is not satisfied for $t=2^{-r+1}$, i.e.,
$$
f(\exp_{p_{k_j}}(2^{-r+1}[-\grad f(p_{k_j})]))> f(p_{k_j})-2^{-r+1}\beta\left\|\grad f(p_{k_j})\right\|^2.
$$
Letting $j$ go to $+\infty$ in the above inequality and taking into account that $\grad f$ and the exponential mapping  are  continuous, we obtain
$$
f(\exp_{\bar{p}}(2^{-r+1}[-\grad f(\bar{p})]))\geq f(\bar{p})-2^{-r+1}\beta\left\|\grad f(\bar{p})\right\|^2.
$$
The last inequality is equivalent to
$$
-\frac{f(\exp_{\bar{p}}(2^{-r+1}[-\grad f(\bar{p})]))-f(\bar{p})}{2^{-r+1}}\leq \beta\left\|\grad f(\bar{p})\right\|^2.
$$
Thus, letting $r$ goes  to $+\infty$ we obtain $\left\|\grad f(\bar{p})\right\|^2\leq \beta\left\|\grad f(\bar{p})\right\|^2$ which implies $\grad f(\bar{p})=0$, i.e., $\bar{p}\in\Omega^*$. Therefore, since $\{p_k\}$ is quasi-Fej\'er convergent to $\Omega^*$, we conclude from 
 Theorem~\ref{teo.qf}   that $\{p_k\}$ converges  to $\bar{p}$ and  the proof is completed.
\end{proof}
\subsection{Iteration-Complexity Analysis} \label{Sec:IteCompAnalysis}
In this section we present an  iteration-complexity bound related to the gradient method  for minimizing a convex functions with  Lipschitz continuous gradient  with constant $L>0$.  In the following,  as an application of  Lemma~\ref{pr:ltd},    we obtain the iteration-complexity  bound for the gradient method with  Strategy~\emph{\ref{adaptive.step}}.
\begin{theorem}\label{teo.complexity.arm}
Let $\{p_k\}$ be    generated by by  Algorithm~\ref{sec:gradient}  with  Strategy~\emph{\ref{adaptive.step}}. Then,   for every $N\in \mathbb{N}$,  there  holds
\begin{equation}\label{complexity.arm}
 f(p_N)-f^*\leq \eta L\frac{L_0 ~d^{2}( p_{0},q)+ 2\left({\cal C}_{\rho,\kappa}^q-1\right)\left[f(p_0)-f^*\right]}{2NL_0},
\end{equation}
for each $q\in \Omega^*$. As a consequence, given  a tolerance $\epsilon>0$,  the number of iterations required to obtain $p_N\in \mathcal{M}$ such that $ f(p_N)-f^*< \epsilon$, is bounded by 
$$
 \eta L \left[ L_0 ~d^{2}( p_{0},q)+ 2\left({\cal C}_{\rho,\kappa}^q-1\right)\left[f(p_0)-f^*\right]\right]/(2L_0\epsilon)=\mathcal{O} \left(1/\epsilon\right).
$$
\end{theorem}
\begin{proof}
Take $q\in \Omega^*$. After  some simple algebraic manipulations and taking into account that $f^*=f(q)$  for each $q\in \Omega^*$,  Lemma~\ref{pr:ltd}  becomes
$$
2t_k\left(f(p_{k+1})-f^*\right)\leq \left[d^2(p_k,q)-d^2(p_{k+1},q)\right] +2t_k\left[{\cal C}_{\rho,\kappa}^q-1\right]\left[f(p_k)-f(p_{k+1})\right],
$$
for all $k=0,1, \ldots$. Using \eqref{des.arm.aux} and taking into account that ${\cal C}_{\rho,\kappa}^q\geq 1$,  $f(p_{k+1})-f^*\geq 0$  and  $f(p_k)-f(p_{k+1}) \geq 0$,  for all $k=0,1, \ldots$, it  follows that 
$$
\frac{2}{\eta L}\left[f(p_{k+1})-f^*\right]\leq \left[d^2(p_k,q)-d^2(p_{k+1},q)\right] +\frac{2}{L_0}\left[{\cal C}_{\rho,\kappa}^q-1\right]\left[f(p_k)-f(p_{k+1})\right],
$$
Summing both sides of the above inequality  for $k= 0,1, \ldots, N-1$, we obtain
$$
\frac{2}{\eta L}\sum_{i=0}^{N-1}\left[f(p_{i+1})-f^*\right]\leq  \left[d^2(p_0,q)-d^2(p_{N},q)\right] +\frac{2}{L_0}\left[{\cal C}_{\rho,\kappa}^q-1\right]\left[f(p_0)-f(p_{N})\right].
$$
Since  $\{f(x_k)\}$ is a decreasing sequence,  we conclude that 
$$
\frac{2}{\eta L}N\left(f(p_{N})-f^*\right)\leq \left[d^2(p_0,q)-d^2(p_{N},q)\right] +\frac{2}{L_0}\left[{\cal C}_{\rho,\kappa}^q-1\right]\left[f(p_0)-f(p_{N})\right],
$$
which  is equivalent to  \eqref{complexity.arm}.  The second  statement of the theorem  follows as  an immediate consequence of the first part.
\end{proof}
Whenever the Lipschitz constant  $L>0$ is computable, we can take a constant stepsize  and  Theorem~\ref{teo.complexity.arm} trivially implies the following result.
 \begin{theorem}\label{th:grad2}    Let $\{p_k\}$ be    generated by  by  Algorithm~\ref{sec:gradient} with    Strategy~\emph{\ref{fixed.step}}.  Then,   for every $N\in \mathbb{N}$,  there holds
\begin{equation} \label{eq:complexity}
 f(p_N)-f^*\leq \frac{L ~d^{2}( p_{0},q)+ 2\left({\cal C}_{\rho,\kappa}^q-1\right)\left[f(p_0)-f^*\right]}{2N}, 
\end{equation}
for each $q\in \Omega^*$. As a consequence, given  a tolerance $\epsilon>0$,  the number of iterations required by the gradient method to obtain $p_N\in \mathcal{M}$ such that $ f(p_N)-f^*< \epsilon$, is bounded by 
$$
 \left[ L ~d^{2}( p_{0},q)+ 2\left({\cal C}_{\rho,\kappa}^q-1\right)\left[f(p_0)-f^*\right]\right]/(2\epsilon)=\mathcal{O} \left(1/\epsilon\right).
$$
\end{theorem}
We remark that, if $\kappa=0$ then  ${\cal C}_{\rho,\kappa}^q=1$.  As a consequence,  Theorem~\ref{th:grad2} merges into  \cite[Theorem 3.2]{BentoFerreiraMelo2017}.
\begin{corollary} \label{cr:icgm}
 Let $\{p_k\}$ be    generated by by  Algorithm~\ref{sec:gradient}  with    Strategy~\emph{\ref{fixed.step}}.   Then,   for every $N\in \mathbb{N}$,  there holds
\begin{equation} \label{eq:complexity2}
\min\left\{\|\grad f(p_{k})\|~:~ k=0, 1,\ldots, N \right\}\leq  \frac{2\sqrt{L\left[L ~d^{2}( p_{0},q)+ 2\left({\cal C}_{\rho,\kappa}^q-1\right)\left[f(p_0)-f^*\right]\right]}}{N},
\end{equation}
for each $q\in \Omega^*$.  As a consequence, given  a tolerance $\epsilon>0$,  the number of iterations required by the gradient method to obtain $p_N\in \mathcal{M}$ such that $ \|\grad f(p_N)\|< \epsilon$, is bounded by $\mathcal{O} (  \sqrt{L\left[L ~d^{2}( p_{0},q)+ 2\left({\cal C}_{\rho,\kappa}^q-1\right)\left[f(p_0)-f^*\right]\right]}/\epsilon)$.
\end{corollary}
\begin{proof}
Let $N\in \mathbb{N}$.  Using the notation ${\lceil N/2 \rceil }$   for    the least integer that is greater than or equal to $N/2$,  we have
\begin{equation} \label{eq:eqcg}
f(p_{N+1})-f^*+ \sum_{j=\lceil N/2\rceil }^{N}\left[ f(p_{j})- f(p_{j+1})  \right]= f(p_{\lceil N/2\rceil })-f^*. 
\end{equation}
Thus,  combining the  last inequality with   Theorem~\ref{th:grad2},  we conclude that 
\begin{equation} \label{eq:eqcgc}
f(p_{N+1})-f^*+ \sum_{j=\lceil N/2\rceil }^{N}\left[ f(p_{j})- f(p_{j+1})  \right]\leq \frac{L ~d^{2}( p_{0},q)+ 2\left({\cal C}_{\rho,\kappa}^q-1\right)\left[f(p_0)-f^*\right]}{2 \lceil N/2\rceil}.
\end{equation}
On the other hand,  using Lemma~\ref{lem:boud} and considering that $t_{k}=1/L$,   we obtain 
$$
 \frac{1}{2L}\sum_{j=\lceil N/2\rceil }^{N}\left\|\grad f(p_j)\right\|^{2} \leq  \sum_{j=\lceil N/2\rceil }^{N}\left[ f(p_{j})- f(p_{j+1})  \right]\leq  f(p_{\lceil N/2\rceil })-f^* .
$$
In view of  $N/2\leq  \lceil N/2\rceil$, the above inequality together with \eqref{eq:eqcg} and \eqref{eq:eqcgc} yield  
$$
\frac{1}{2L}\sum_{j=\lceil N/2\rceil }^{N}\left\|\grad f(p_j)\right\|^{2} \leq \frac{L ~d^{2}( p_{0},q)+ 2\left({\cal C}_{\rho,\kappa}^q-1\right)\left[f(p_0)-f^*\right]}{N}.
$$
Therefore, 
 $$
\min\{\|\grad f(p_{k})\|^2~;~ k=\lceil N/2 \rceil , \ldots, N\}\leq  \frac{4L\left[L ~d^{2}( p_{0},q)+ 2\left({\cal C}_{\rho,\kappa}^q-1\right)\left[f(p_0)-f^*\right]\right]}{N^2}, 
 $$
which implies the  desired inequality.  The second  statement of the corollary follows as  an immediate consequence of the first one.
\end{proof}
We  end this section  by recalling  an  iteration-complexity bound for  non-convex functions defined in a general Riemannian manifolds, which   appeared  in \cite{BoumalAbsilCartis2016}. 
\begin{theorem}\label{th:grad1} 
Let $\{p_k\}$ be    generated by  by  Algorithm~\ref{sec:gradient} with    Strategy~\emph{\ref{fixed.step}}.  Then,   for every $N\in \mathbb{N}$,  there  holds
$$
\min \left\{\|\grad f(p_{k})\|~:~ k=0, 1,\ldots, N \right\}\leq \frac{\sqrt{ 2L(f(p_0) -f^*)}}{\sqrt{N+1}}.  
$$
As a consequence, given  a tolerance $\epsilon>0$,  the number of iterations required to obtain $p_N\in \mathcal{M}$ such that $\|\grad f(p_{N})\|<\epsilon$ is bounded by  ${\cal O} ( L(f(p_0) -f^*)/\epsilon^2)$.
\end{theorem}
Under the assumption  of convexity and lower  boundedness of curvature,  we conclude that  Corollary~\ref{cr:icgm} improves  Theorem~\ref{th:grad1}. It is worth to point out that  results  on iteration-complexity bound   to the gradient method on Riemannian manifold with non-negative curvature and in  Hadamard manifolds  with lower bound curvature has already appeared \cite{BentoFerreiraMelo2017,ZhangReddiSra2016,ZhangSra2016}. The result of this section present  a contribution to the systematic study of the iteration-complexity of  the gradient  methods in the Riemannian setting.
\section{Examples}\label{Sec:Examples}
In the following  sections, we present some  examples of functions satisfying the assumptions of our  results in the previous sections.  In particular  we show that, by endowing the   constrained set  with a suitable  Riemannian metric, a constrained  Euclidean optimization problem  with non-convex objective function  having  non-Lipschitz gradient can be seen as unconstrained  Riemanian optimization problem with  convex objective   function  having  Lipschitz gradient.    {\it Throughout the next  sections we denote 
 $$
 \mathbb{R}_{++}^{n}\coloneqq \left\{x\coloneqq (x_1, \ldots, x_n)^{T}\in\mathbb{R}^{n\times 1}~:~~x_i>0,~~i=1,\ldots,n\right\}, 
 $$
 the positive orthant,  ${\mathbb P}^{n}$  the set of  symmetric matrices of order $n\times n$ and  ${\mathbb P}^{n}_{++}$  the cone of  symmetric positive definite matrices}.

\subsection{Examples in the positive orthant}\label{Sec:ExamplesPO}
In this section,  we present  examples in the  positive orthant endowed with a new Riemannian metric.  To present this examples we need some definitions and results  of Riemannian geometry.  Endowing $\mathbb{R}_{++}^n$ with the Riemannian metric $\left\langle \cdot,\cdot\right\rangle$ defined by $\left\langle u,v\right\rangle\coloneqq u^TG(x)v$,   for all $x\in\mathbb{R}_{++}^{n}$ and  $u, v \in T_{x} \mathbb{R}_{++}^n \equiv\mathbb{R}^n$,  where $G : \mathbb{R}_{++}^n \to {\mathbb P}^{n}_{++}$ is given by 
\begin{equation} \label{eq:MetricPO}
G(x)\coloneqq \mbox{diag}(x^{-2}_{1},\ldots,x^{-2}_{n})\in \mathbb{R}^{n\times n}, 
\end{equation}
we obtain  a  complete Riemannian manifold with   zero curvature, which will be denoted by   ${\mathcal M}\coloneqq (\mathbb{R}_{++}^n,G)$.  Let $f: {\mathcal M} \to \mathbb{R}$ be a  twice differentiable function.  We denote by  $f'(x)$ and  $f''(x)$   the  Euclidean gradient and  hessian of $f$ at $x$,  respectively. Thus,  \eqref{eq:MetricPO} implies  that the Riemannian {\it gradient} and {\it hessian} of  $f$  are given, respectively,  by
\begin{align}  
\mbox{grad} f(x)&= \mbox{diag} (x)^{2}f'(x),    \qquad x\in{\mathcal M},\label{eq:GradExPO} \\
\mbox{hess}\,f(x)v&=  \left[\mbox{diag}(x)^2 f''(x)+ \mbox{diag}(x)  \mbox{diag}\left(f'(x)\right)\right]v,     \qquad x\in{\mathcal M}, \label{eq:HessExPO}
\end{align}
 where $\mbox{diag} (x)\coloneqq \mbox{diag}(x_{1},\ldots,x_{n})\in \mathbb{R}^{n\times n}$.   Next we present two examples of    convex functions with Lipschitz gradient in  ${\mathcal M}\coloneqq (\mathbb{R}_{++}^n,G)$.
\begin{example} \label{Ex:Log2Log}
 Consider the function  $f:\mathbb{R}_{++}^n\to\mathbb{R}$ defined by  
\begin{equation} \label{eq:Log2Log}
f(x)\coloneqq \sum_{i=1}^nf_i(x_i), \qquad f_{i}(x_i)\coloneqq -a_i\mathrm{e}^{-b_ix_i}+c_i\ln\left(x_i\right)^2+d_i\ln\left(x_i\right),  \quad i=1,\ldots n, 
\end{equation}
where  $a_i,b_i,d_i\in{\mathbb{R}_{+}}$ and $c_i\in{\mathbb{R}_{++}}$ satisfy  $c_i>a_i$. Since $f$ is coercive, it has a minimum.  By using \eqref{eq:Log2Log}  the  first  and second derivative of $f$ at $x\in \mathbb{R}_{++}^n$ are   given by $f'(x)\coloneqq  \left(f'_{1}(x_1), \ldots, f'_{n}(x_n)\right)$ and $f''(x)\coloneqq  \emph{diag}\left(f''_1(x_1), \ldots, f''_n(x_n)\right)$, where
\begin{equation}\label{eq:fhmx2s.ex2}
 f'_{i}(x_i)=a_ib_i\mathrm{e}^{-b_ix_i}+2c_i\frac{\ln(x_i)}{x_i}+\dfrac{d_i}{x_i}, \quad f''_{i}(x_i)=-a_ib_i^2\mathrm{e}^{-b_ix_i}+2c_i\left[\frac{1-\ln(x_i)}{x_i^2}\right]-\dfrac{d_i}{x_i^2}, 
\end{equation}
for all $ i=1, \dots, n$. Note that $f''_{i}(1)<0$, for all $ i=1, \dots, n$,  and then   $f$ is not Euclidean convex. 
Using   \eqref{eq:HessExPO} and \eqref{eq:fhmx2s.ex2}  the  hessian of  $f$ in ${\mathcal M}\coloneqq (\mathbb{R}_{++}^n,G)$ is given by
$$
\emph{Hess}\,f(x)v\coloneqq \left(a_1b_1\mathrm{e}^{-b_1x_1}\left(x_1-b_1x_1^2\right)+2c_1, \dots, a_nb_n\mathrm{e}^{-b_nx_n}\left(x_n-b_nx_n^2\right)+2c_n \right)v.
$$
Since $c_i>a_i$,  we have $a_ib_i\mathrm{e}^{-b_ix_i}\left(x_i-b_ix_i^2\right)+2c_i\geq0$, for all   $ i=1, \dots, n$.  Hence,  by using the definition of the metric,   for  $ v=(v_1, \ldots, v_n)^{T}\in \mathbb{R}^n$ and $x\in{\mathbb{R}_{++}}$,  we have 
$$
\left\langle \emph{Hess}\,f(x)v, v\right\rangle=\sum_{i=1}^n\left[a_ib_i\mathrm{e}^{-b_ix_i}\left(x_i-b_ix_i^2\right)+2c_i\right]\frac{v_i^2}{x_i^2}\geq 0, 
$$
concluding that  $f$ is convex in ${\mathcal M}$.  Since $\|v\|=v^TG(x)v= 1$,  we have $v_{i}^2\leq x_{i}^2$ and  owing  that $\left(a_ib_i\mathrm{e}^{-b_ix_i}\left(x_i-b_ix_i^2\right)+2c_i\right)< a_i+2c_i$, for all  $i=1, \ldots, n$ ,  we obtain 
$$
\left\|\emph{Hess}\,f(x) v\right\|^2=\sum_{i=1}^n\left[a_ib_i\mathrm{e}^{-b_ix_i}\left(x_i-b_ix_i^2\right)+2c_i\right]^2\frac{v_i^2}{x_i^2}< \sum_{i=1}^n(a_i+2c_i)^2,  \qquad x\in{\mathbb{R}_{++}}.
$$
Therefore,  \eqref{eq:DefNormHess} and  Lemma~\ref{le:CharactGL}  imply  that $\grad f$ is Lipschitz with  $L< \sum_{i=1}^n(a_i+2c_i)^2$. 
\end{example}
\begin{example} \label{Ex:LogLog}
Consider the function  $f:\mathbb{R}_{++}^n\to\mathbb{R}$ defined by  
\begin{equation} \label{eq:LogLog}
f(x)\coloneqq \sum_{i=1}^nh_i(x_i), \qquad f_{i}(x_i)\coloneqq a_i\ln\left(x^{d_i}_i+b_i\right) - c_i \ln\left(x_i\right),  \qquad i=1,\ldots n, 
\end{equation}
where $a_i,b_i,c_i,d_i\in{\mathbb{R}_{++}}$ satisfy   $c_i< a_id_i$ and $d_i\geq2$, for all $i=1,\ldots n$. The minimizer of $f$  is $x^*=(x^*_1, \dots, x^*_n)$, where $x^*_i=\sqrt[d_i]{b_ic_i/(a_id_i-c_i)}$, for  $i=1, \dots, n$. Function $f$ in  \eqref{eq:LogLog} is not Euclidean convex. However,  by following the same steps  as in the Example~\ref{Ex:Log2Log},  we can show that $f$ is convex  and  has gradient  Lipschitz   with constant  $L<\sum_{i=1}^na^2_id_i^4$   in ${\mathcal M}= (\mathbb{R}_{++}^n,G)$.
 \end{example}
  We end this section by presenting,  without giving the details,  two more examples of  convex  functions with   Lipschitz gradients   in ${\mathcal M}\coloneqq (\mathbb{R}_{++}^n,G)$.
 \begin{remark}
  Let $a,b,c\in\mathbb{R_{++}}$. Define  $h_1:\mathbb{R}_{++}^n\to\mathbb{R}$ by  $h_1(x)\coloneqq a\ln\left(x^Tx+b\right) -c \ln\left(x_1\ldots x_n\right),$ where   $n c< 2a$,  and $h_2:\mathbb{R}_{++}^n\to\mathbb{R}$   by $h_2(x)=a\ln\left((x_1\ldots x_n)^2+b\right) -c \ln\left(x_1\ldots x_n\right)$.  By using similar arguments   of Examples~\ref{Ex:Log2Log},   we can prove that   $h_1$ and  $h_2$  are also   convex  with Lipschitz gradient in the Riemannian manifold ${\mathcal M}= (\mathbb{R}_{++}^n,G)$.
 \end{remark}

\subsection{Examples in the SPD matrices cone}\label{Sec:ExampleSDP}
In this section,  we present   examples in the cone of symmetric positive definite matrices with new Riemannian metric.  Following Rothaus~\cite{Rothaus1960}, let $\mathcal{M}\coloneqq ({\mathbb P}^n_{++}, \langle \cdot , \cdot \rangle)$ be the Riemannian manifold endowed with the Riemannian metric given by  
\begin{equation}\label{eq:metric}
\langle U,V \rangle\coloneqq \mbox{tr} (VX^{-1}UX^{-1}),\qquad X\in \mathcal{M}, \qquad U,V\in
T_X\mathcal{M},
\end{equation}
where $\mbox{tr}(X)$ denotes the trace of  $X\in {\mathbb P}^n$ and $T_X\mathcal{M}\approx\mathbb{P}^n$.  In fact,  $\mathcal{M}$ is a Hadamard manifold,  see  for example \cite[Theorem 1.2. p. 325]{Lang1999} and  its curvature  is bound below;  see \cite{LengletRoussonDericheFaugeras2006}.  The {\it gradient} and {\it hessian} of $f:{\mathbb P}^n_{++}\longrightarrow\mathbb{R}$ are  given by
\begin{align} 
\mbox{grad} f(X)&=Xf'(X)X, \label{eq:Grad}\\
\mbox{hess}\,f(X)V&=Xf''(X)VX+\frac{1}{2}\left[  Vf'(X)X+  Xf'(X)V \right] \label{eq:Hess},
\end{align}
where $V\in T_X\mathcal{M}$, $f'(X)$ and  $f''(X)$ are the  Euclidean gradient and  hessian of  $f$ at $X$, respectively.  In the following,   we present two examples of    convex functions with   Lipschitz gradient in  $\mathcal{M}\coloneqq ({\mathbb P}^n_{++}, \langle \cdot , \cdot \rangle)$.
\begin{example} \label{ex:psdms}
Consider the function    $f:{\mathbb P}^n_{++}\longrightarrow\mathbb{R}$ defined by 
 \begin{equation}  \label{eq:fpdm4}
f(X)= a\ln(\det(X))^2-b\ln\left(\det(X)\right). 
\end{equation}
where $a,b\in{\mathbb{R}_{++}}$.  The  Euclidean gradient and hessian   of  $f$ are  given, respectively,  by
\begin{align} 
f'(X) &= \left[2a\ln(\det(X))-b\right]X^{-1},  \label{eq:GradE4}\\
f''(X)V &= 2a\,\emph{tr}(X^{-1}V)  X^{-1}-\left[2a\ln(\det(X))-b\right]X^{-1}VX^{-1},   \label{eq:HessE4}
\end{align}
for all $X\in {\mathbb P}^n_{++}$ and $V\in{\mathbb P}^n$.  
It follows from  \eqref{eq:GradE4}  that each $X\in \mathcal{M}$ satisfying  $\det X=e^{b/(2a)}$ is a critical point of $f$.  Thus, letting $V=I_{n}$ and $X=tI_{n}$ with $t\in{\mathbb{R}_{++}}$  in   \eqref{eq:HessE4}  we obtain that 
$
f''(tI_n)I_n= [2a n t^{-2}-2an\ln t+b]I_n.
$
Thus  $f''(tI_n)$ is not positive definite for $t$ sufficiently large. Hence, $f$ is not Euclidean  convex. Moreover,   $f''$ is not bounded and consequently $f'$ is not Lipschitz. On the other hand, combining  \eqref{eq:Hess}, \eqref{eq:GradE4} and \eqref{eq:HessE4}, after some calculation we obtain \begin{equation}\label{eq:HessR4}
\emph{Hess}\,f(X)V= 2a\,\emph{tr}(X^{-1}V)X, \qquad \langle\emph{Hess}\,f(X)V, V\rangle=2a\emph{tr}(X^{-1}V)^2\geq 0,
\end{equation}
for all $X\in \mathcal{M}$ and $V\in T_X\mathcal{M}$. Thus,   $f$ is  convex in $\mathcal{M}$. Moreover,   \eqref{eq:metric} with  \eqref{eq:HessR4} yield  $\|\emph{Hess}\,f(X)V\| =2a\emph{tr}(X^{-1}V)$, for all $X\in \mathcal{M}$ and $V\in T_X\mathcal{M}.$ If we assume that $\|V\|^2=\emph{tr} (VX^{-1}VX^{-1})=1$ then $\emph{tr}(X^{-1}V)\leq \sqrt{n}$. Hence, 
$$
\|\emph{Hess}\,f(X)V\|\leq2a  \sqrt{n},  \qquad X\in\mathcal{M}, \qquad V\in T_X\mathcal{M}, \quad \|V\|=1.
$$
 Therefore,   \eqref{eq:DefNormHess} and  Lemma~\ref{le:CharactGL}   imply  that  $\grad f$ is Lipschitz with constant  $L\leq 2a  \sqrt{n}$.   
\end{example}
\begin{example} \label{ex:psdm}
 Consider  the function  $f:{\mathbb P}^n_{++}\longrightarrow\mathbb{R}$ defined by 
\begin{equation}  \label{eq:fpdm}
f(X)= a\ln\left(\det(X)^{b_1}+b_2\right) -c\ln\left(\det X\right), 
\end{equation}
where  $a,b_1, b_2,c\in{\mathbb{R}_{++}}$ with  $c< ab_1$.    Function $f$ in  \eqref{eq:fpdm}  is not Euclidean convex. On the other hand,  by using similar arguments  as in the Example~\ref{ex:psdms},  we can see  that $f$ is convex  and  has Lipschitz  gradient  with constant  $L<ab_1^2n$   in $\mathcal{M}=({\mathbb P}^n_{++}, \langle \cdot , \cdot \rangle)$. 
\end{example}
\section{Numerical Experiments}\label{Sec:NumExp}
In this section,  we present some numerical experiments to illustrate the behavior of the Riemannian gradient method for minimizing convex functions onto the positive orthant and the cone of symmetric positive definite matrices. We implemented Algorithm~\ref{sec:gradient} with Strategies~1, 2 and 3, and tested it on the examples of Section~\ref{Sec:Examples}. Additionally, we consider the application of the method to compute the Riemannian center of mass, which is a specific instance of a geometric mean for points in a Riemannian manifold. In due course, we will describe this problem in more detail.

For the positive orthant, the {\it exponential mapping} $\exp_x:T_x\mathcal{M}\to \mathcal{M}$ in the Riemannian manifold ${\mathcal M}\coloneqq(\mathbb{R}_{++}^n,G)$ is assigned by 
\begin{equation} \label{eq:ExponentialExPO}
\exp_x(v) = \left(x_1e^{\frac{v_1}{x_1}},\ldots,x_ne^{\frac{v_n}{x_n}}\right),
\end{equation}
for each  $v:=(v_1,\ldots,v_n)^{T}\in\mathbb{R}^{n\times 1}$ and $x\coloneqq (x_1, \ldots, x_n)^{T}\in \mathbb{R}_{++}^n$,  see~\cite{NesterovTodd2002}. By using the gradient in   \eqref{eq:GradExPO} and the definition of metric  we obtain
$$
\left\|\mbox{grad} f(x)\right\|^{2}=  \mbox{grad} f(x)^TG(x)\mbox{grad} f(x) = \sum_{i=1}^n\left[x_i \frac{\partial f}{\partial x_i}(x)\right]^2,
$$
for each  $x\coloneqq(x_1,\ldots,x_n)\in{\mathcal M}$. Considering the cone of symmetric positive definite matrices, the {\it exponential mapping} $\exp_X:T_X\mathcal{M}\to \mathcal{M}$ in the Riemannian manifold $\mathcal{M}\coloneqq ({\mathbb P}^n_{++}, \langle \cdot , \cdot \rangle)$,  is given by
\begin{equation} \label{eq:GoedSPD}
\exp_X(V)=X^{1/2}e^{\left(X^{-1/2}VX^{-1/2}\right)}X^{1/2}, 
\end{equation}
for each $V\in {\mathbb P}^n$ and $X\in {\mathbb P}^n_{++}$. By using \eqref{eq:Grad},  we have   $\left\|\mbox{grad} f(X)\right\|^2= \mbox{tr} \left( \left[Xf'(X)\right]^2\right),$
for each $X\in \mathcal{M}$. In both cases, although~\eqref{eq:OptP} is a constrained optimization problem, by~\eqref{eq:ExponentialExPO} and \eqref{eq:GoedSPD}, Algorithm~\ref{sec:gradient} generates only feasible points without using projections or any other procedure to remain the feasibility. Hence, problem~\eqref{eq:OptP} can be seen as unconstrained Riemannian optimization problem.

Our numerical experience indicates that it is advantageous to perform a reasonably stringent line search. Therefore, we used $\beta=1/2$ for Strategies~2 and 3. Additionally, we set $L_0=1$ and $\eta=2$ for Strategy~2. We stopped the execution of Algorithm~\ref{sec:gradient} at $p_k$ declaring convergence if
$$\|f^{\prime}(p_k)\|_{\infty}\leq 10^{-5}.$$
Since, by~\eqref{eq:GradExPO} and~\eqref{eq:Grad}, $\mbox{grad} f(p_k)=0$ if only if $f^{\prime}(p_k)=0$, this is a reasonable stopping criterion. The maximum number of allowed iterations was set to 1000. Codes are written in Matlab and are freely available at \url{https://orizon.mat.ufg.br/}.

\subsection{Academic problems} \label{Sec:NumExpAcademic}

We begin the numerical experiments by testing the Riemannian gradient method on the problems of minimizing the functions of the examples in Sections~\ref{Sec:Examples}. We call these problems by Problem~1, 2, 3 and 4, respectively.

\subsubsection{Academic problems in the positive orthant} \label{Sec:NumExpPO}

In this section, we compare the performance of the Riemannian with the Euclidian gradient methods on Problems~1 and~2. We considered Algorithm~\ref{sec:gradient} with Strategy~3 and implemented the Euclidian gradient method also using the Armijo rule with the same algorithmic parameters. It is worth mentioning that, in principle, the Euclidian method can generate iterates out of the positive orthant. Thus, in order to keep the feasibility, in each iteration we simply determine the maximum step size to remain within the feasible set and perform a convenient linear search by shrinking the step size until the Armijo condition is satisfied.

We generated several instances of Problems~1 and~2 by considering functions~\eqref{eq:Log2Log} and~\eqref{eq:LogLog}, respectively, with $n=100$ and different parameters. In all cases, for each $i=1,\ldots,n$, we set  parameters $a_i$ with the same value. Equivalently for parameters $b_i$, $c_i$, and $d_i$.

\vspace{0.5cm}
\noindent{\bf Problem 1.} First, parameters $a_i$, $b_i$, and $d_i$ were randomly generated between $0$ and $10$. Then, in order to guarantee that $c_i> a_i$, we randomly generated parameters $c_i$ between $1.1 a_i$ and $5.0 a_i$. All problems were solved 100 times using starting points from a uniform random distribution inside the box $[0, \; 20]^n$. For each method, Table~\ref{tab:prob1} informs the percentage of runs that has reached a critical point ($\%$), the average numbers of iterations (it) and functions evaluations (nfev) of the successful runs.

 \begin{table}[htb!] 
  {\footnotesize
  \centering
\begin{tabular}{|c|cccc|ccc|ccc|} \hhline{~~~~~*6{-}|}  
  \multicolumn{5}{c|}{}& \multicolumn{3}{c|}{\cellcolor[gray]{0.9} Riemannian} & \multicolumn{3}{c|}{\cellcolor[gray]{0.9} Euclidian}\\ 
  \multicolumn{5}{c|}{}& \multicolumn{3}{c|}{\cellcolor[gray]{0.9} Gradient method} & \multicolumn{3}{c|}{\cellcolor[gray]{0.9} Gradient method}\\ \hline \rowcolor[gray]{0.9}
$\#$ &$a_i$ & $b_i$ & $c_i$ & $d_i$ & \% & it & nfev  & \% & it & nfev  \\ \hline 
1  & 3.77 &  8.17 & 11.10 &  5.92 & 100.0 &  14.1 &   85.5 & 100.0 &  72.3  &  255.4  \\ \hline 
2  & 7.88 &  5.49 & 17.95 &  3.01 & 100.0 &  21.1 &  148.5 & 100.0 &  56.9  &  208.2  \\ \hline 
3  & 8.96 &  1.72 & 42.11 &  7.18 & 100.0 &  17.0 &  137.1 & 100.0 &  56.0  &  203.3  \\ \hline 
4 & 3.14 &  1.30 & 13.77 &  9.32 &  100.0 &  9.0 &   55.0 &  100.0 & 76.3  &  232.6  \\ \hline 
5  & 5.49 &  1.72 &  6.82 &  0.83 & 100.0 &  10.0 &   51.0 & 100.0 &  65.1  &  227.3  \\ \hline 
6  & 4.59 &  4.25 & 13.31 &  8.11 & 100.0 &  11.0 &   67.0 & 100.0 &  71.2  &  228.5  \\ \hline 
7  & 2.10 &  3.80 &  4.31 &  0.10 & 100.0 &  21.2 &  107.0 & 100.0 &  54.1  &  184.7  \\ \hline 
8  & 8.69 &  7.47 & 28.54 &  4.77 & 100.0 &  8.0 &   57.1 &  100.0 & 61.1  &  255.3  \\ \hline 
9  & 9.85 &  2.24 & 44.60 &  0.57 & 100.0 &  16.0 &  129.0 & 100.0 &  52.0  &  201.9  \\ \hline 
10 & 2.60 &  1.71 &  9.65 &  2.07 & 100.0 &  18.0 &  109.2 & 100.0 &  57.1  &  185.6  \\ \hline 
11 & 6.03 &  1.40 & 13.57 &  8.94 & 100.0 &   9.0 &   55.0 & 100.0 &  79.2  &  238.1  \\ \hline 
12 & 5.71 &  4.99 &  9.37 &  3.22 & 100.0 &  20.1 &  121.7 & 100.0 &  59.2  &  191.2  \\ \hline 
13 & 1.38 &  6.07 &  6.78 &  4.86 & 100.0 &   9.0 &   46.1 & 100.0 &  73.5  &  219.6  \\ \hline 
14 & 2.22 &  0.24 &  5.58 &  9.04 & 100.0 &  14.0 &   71.0 & 100.0 &  141.8 &  408.6  \\ \hline 
15 & 4.19 &  6.24 &  7.73 &  9.48 & 100.0 &   7.0 &   36.0 & 100.0 &  105.8 &  315.4  \\ \hline 
16 & 8.27 &  2.42 & 10.96 &  3.02 & 100.0 &  17.0 &  103.0 & 100.0 &  66.3  &  237.4  \\ \hline 
17 & 4.72 &  0.64 & 19.35 &  0.62 & 100.0 &  18.0 &  127.0 & 100.0 &  55.6  &  204.1  \\ \hline 
18 & 2.99 &  1.63 & 11.15 &  6.44 & 100.0 &  14.0 &   85.1 & 100.0 &  75.8  &  250.8  \\ \hline 
\end{tabular}
\caption{Parameters of function~\eqref{eq:Log2Log} as well as the performance of the Riemannian and Euclidian gradient methods.}
\label{tab:prob1}}
 \end{table} 
 
As can be seen, the Riemannian gradient method is clearly more efficient than the Euclidian gradient method in this set of problems. In {\it all} 18 problem instances considered, the Riemannian version required fewer iterations and function evaluations. Overall, on average, the Riemannian gradient method performed $19.8\%$ of iterations and $37.5\%$ of function evaluations required by the Euclidian method.


Figure~\ref{fig:prob1}~(a) shows a typical behavior of the methods on Problem~1. This case corresponds to $n=2$, $a_i=1$, $b_i=c_i=d_i=2$ for $i=1,2$, and the initial point $p_0=[5, \; 1]^T$. The stopping criterion was satisfied with 4 and 14 iterations for the Riemannian and Euclidian gradient methods, respectively. The {\it zig-zag} path of the Euclidian gradient method can be seen clearly. In contrast, the Riemannian method rapidly approaches the minimizer. In Figure~\ref{fig:prob1}~(b), the sup-norm of the euclidean gradient is displayed as a function of the iteration number, which clearly shows the distinction between the methods. While the Euclidian method required 10 iterations for $\|f^{\prime}(p_k)\|_{\infty}$ to reach order of $10^{-2}$, the Riemannian algorihtm required only 2 iterations. 

 \begin{figure}[h!]
\begin{minipage}[b]{0.50\linewidth}
\begin{figure}[H]
	\centering
		\includegraphics[scale=0.5]{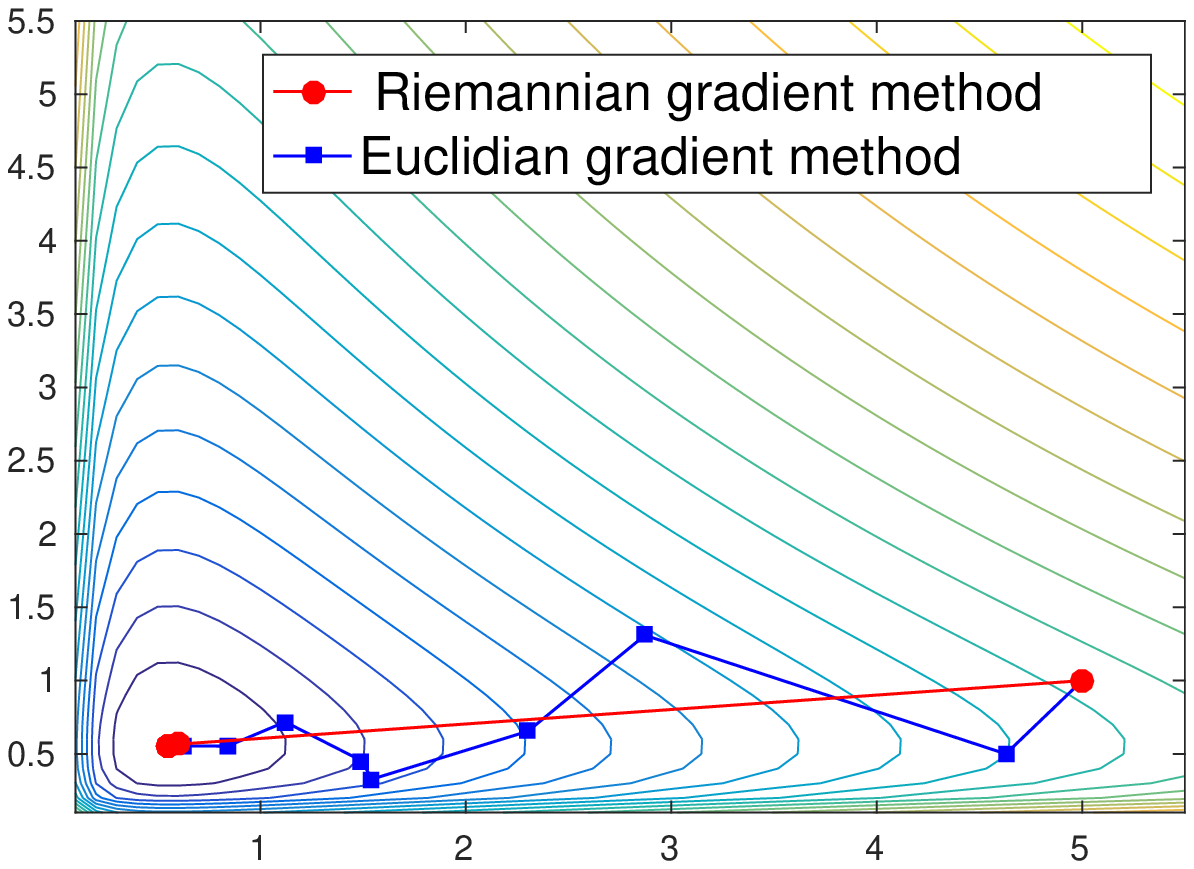}\\
	\footnotesize	(a)
\end{figure}

\end{minipage} \hfill
\begin{minipage}[b]{0.50\linewidth}

\begin{figure}[H]
	\centering
		\includegraphics[scale=0.5]{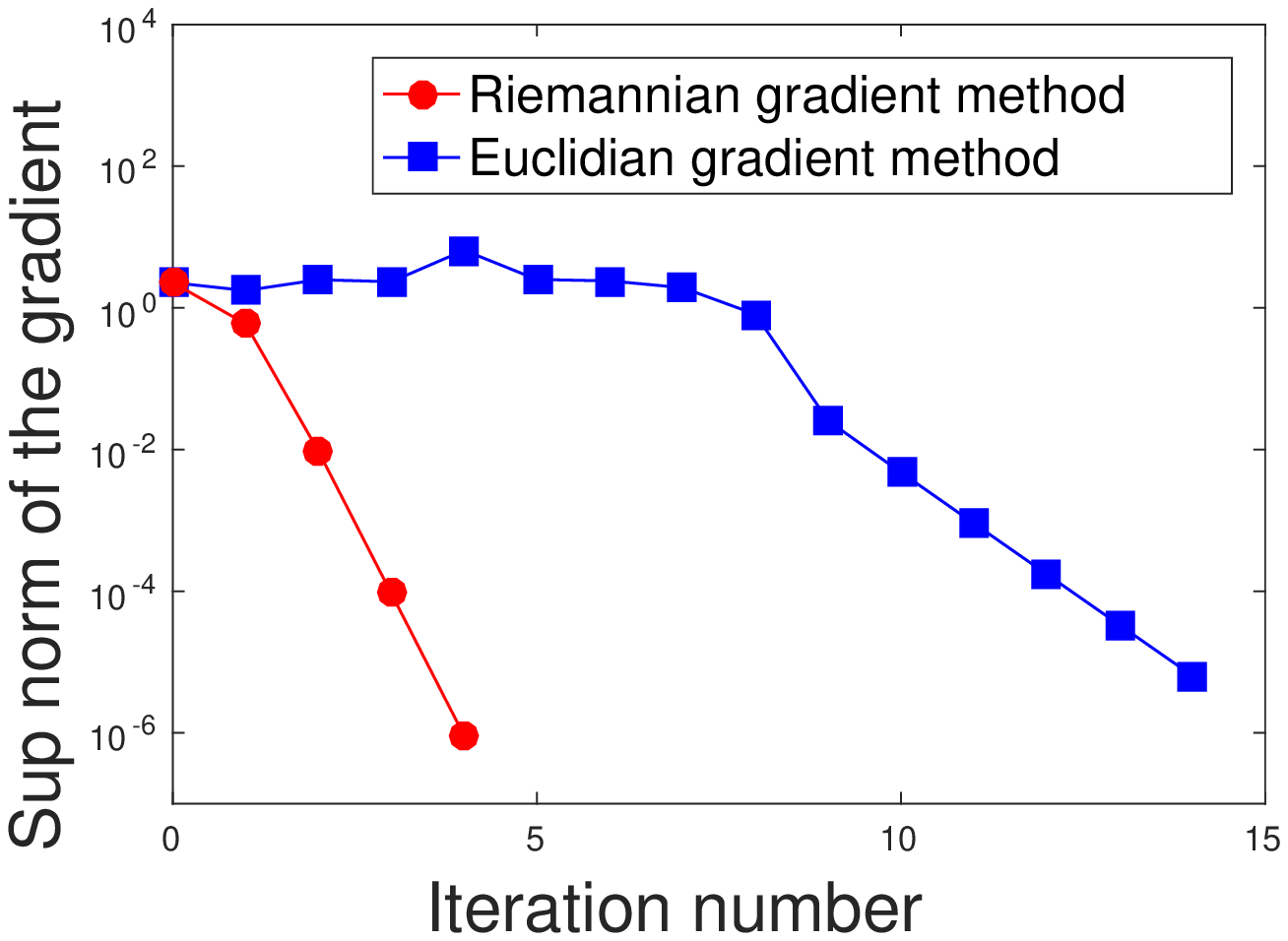}\\
	\footnotesize	(b)
\end{figure}
\end{minipage}\hfill
\caption{(a) A typical behavior of the Riemannian and the Euclidian gradient methods for which the {\it zigzag} pattern appears for the Euclidian algorihtm. (b) Sup-norm of the euclidean gradients per iteration.}
\label{fig:prob1}
\end{figure}

\vspace{0.5cm}
\noindent{\bf Problem 2.} We tested the algorithms on a set of 100 instances of Problem~2. We randomly generated parameters $a_i$ and $b_i$ between $0$ and $10$, parameters $d_i$ between $2$ and $10$, and a constant $\mu_i$ belonging to the interval $(0, \; 1)$. Then, we set $c_i = \mu_i a_i d_i$, fulfilling the conditions  $c_i< a_id_i$ and $d_i\geq2$, for all $i=1,\ldots,n$. As for Problem~1, each instance was solved 100 times using starting points from a uniform random distribution inside the box $[0,  \; 20]^n$. The results are given in the following form: for each problem instance, Figure~\ref{fig:prob2}~(a) informs the average number of iterations, and Figure~\ref{fig:prob2}~(b) informs the average number of functions evaluations. As a matter of aesthetics, the graphs are independent and were organized in an increasing way. 

 \begin{figure}[h!]
\begin{minipage}[b]{0.50\linewidth}
\begin{figure}[H]
	\centering
		\includegraphics[scale=0.5]{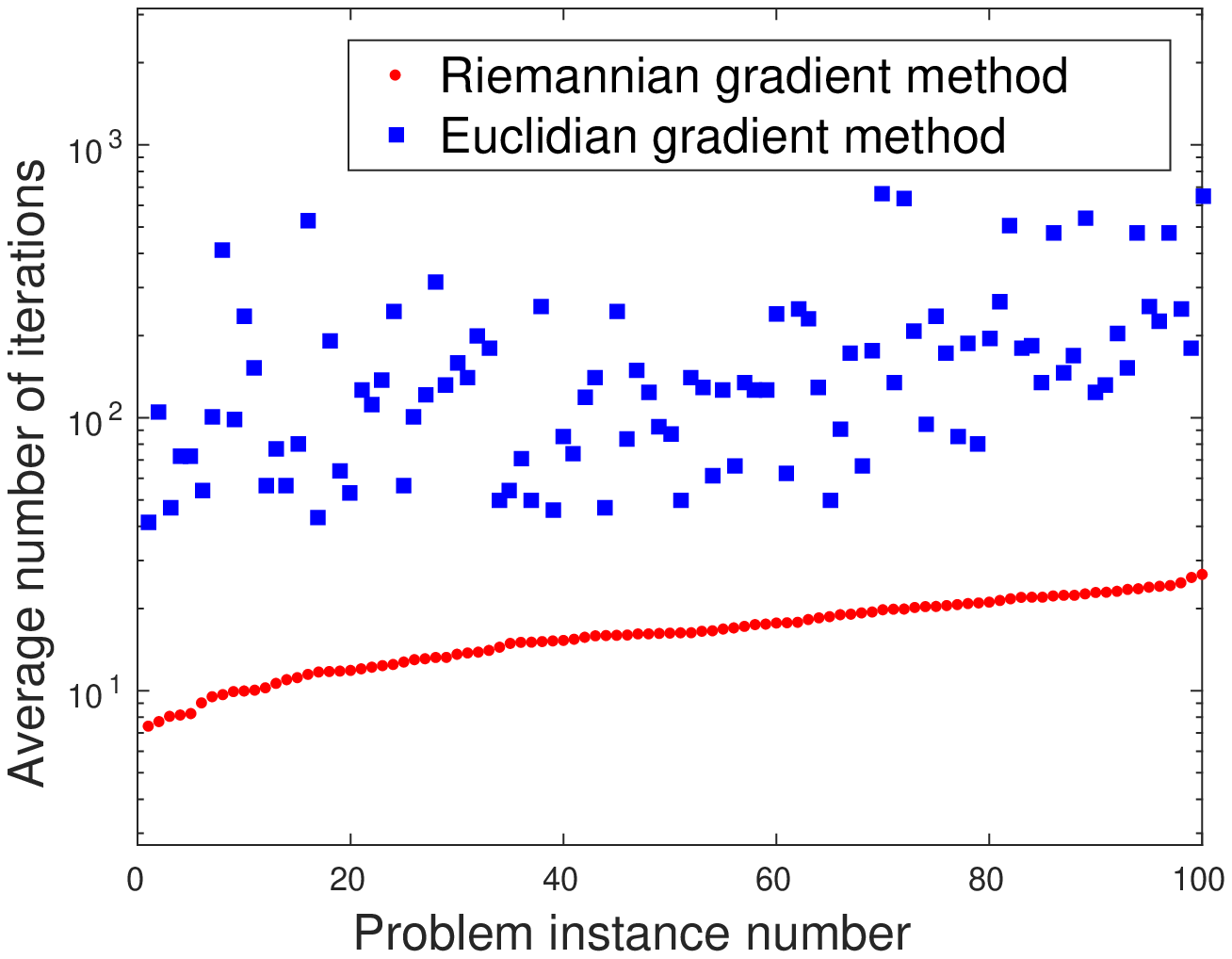}\\
	\footnotesize	(a)
\end{figure}

\end{minipage} \hfill
\begin{minipage}[b]{0.50\linewidth}

\begin{figure}[H]
	\centering
		\includegraphics[scale=0.5]{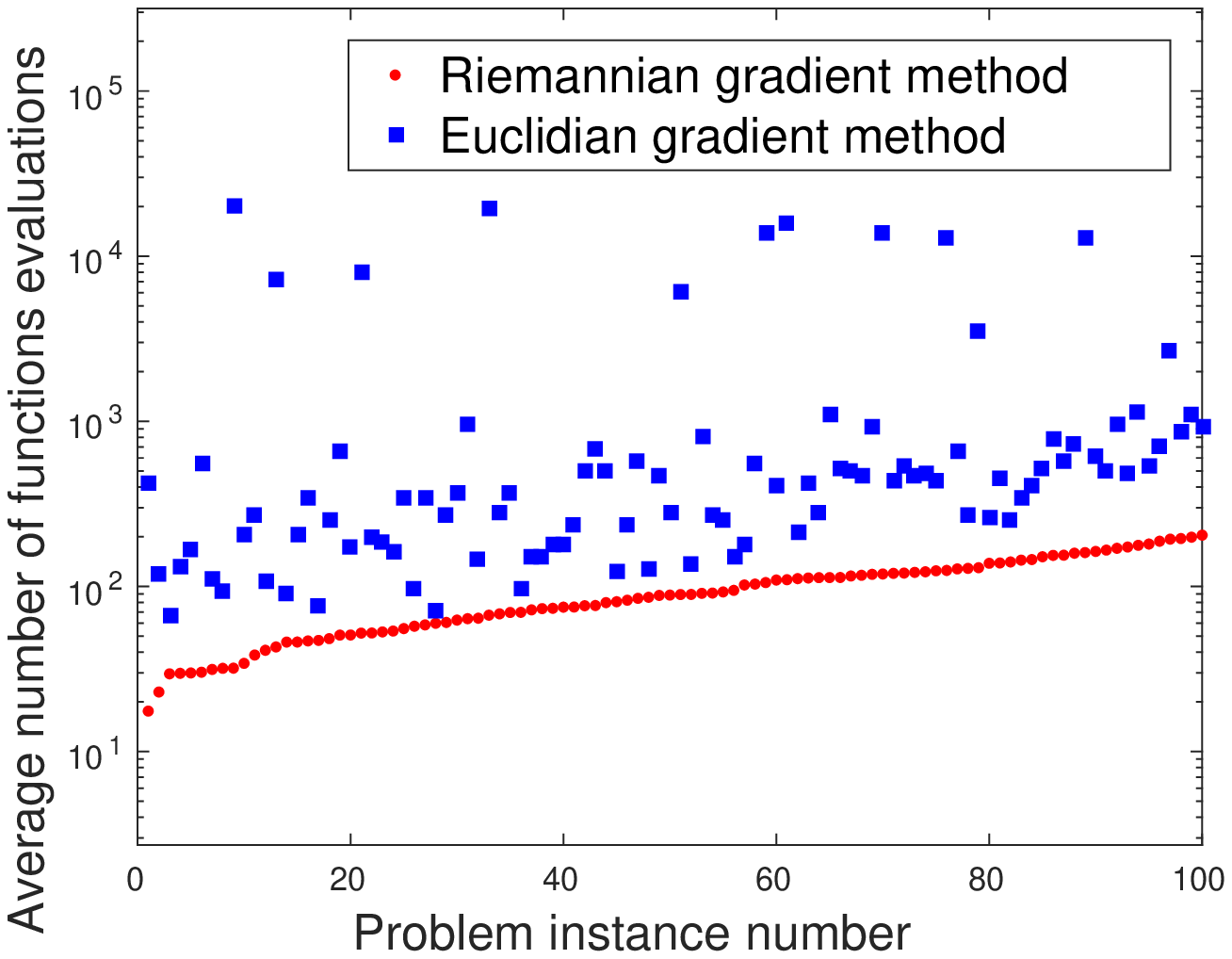}\\
	\footnotesize	(b)
\end{figure}
\end{minipage}\hfill
\caption{(a) Average number of iterations and (b) average number of functions evaluations required for each of 100 instances of Problem~2 for the Riemannian and the Euclidian gradient methods.}
\label{fig:prob2}
\end{figure}

Figure~\ref{fig:prob2} shows that the Riemannian gradient method required fewer iterations and function evaluations than the Euclidian gradient method in {\it all} problem instances. In terms of percentages, on average, the Riemannian algorihtm performed $9.7\%$ and $5.6\%$ of the number of iterations and functions evaluations required by the Euclidian algorithm, respectively.

The results of this section allow us to conclude that there are problems for which the introduction of a suitable metric makes it possible to explore its geometric and algebraic structures, resulting in a large reduction in the computational cost of obtaining its solution. In fact, by introducing a suitable  Riemannian metric, a constrained optimization problem with non-convex objective function and non-Lipschitz gradient can be transformed into an  optimization problem with convex objective function and Lipschitz gradient. 



\subsubsection{Academic problems in the SPD matrices cone}\label{Sec:NumExpSDP}

In this section we illustrate the practical applicability of the Riemannian gradient method for minimizing convex functions onto the cone of symmetric positive definite matrices. 
We used Problem~3 to test the Riemannian gradient method varying the dimension and the domain of the starting points, while Problem~4 was used to compare the different linear search strategies. For Problem~3, we adopted Strategy~3.

\vspace{0.5cm}
\noindent{\bf Problem 3.} We set $a=b=1$ in function~\eqref{eq:fpdm4}. In the first set of tests, we assigned the following values to the dimension: $n=10$, $20$, $50$, $100$, and $150$. For each specific value of $n$, we run the Riemannian gradient method 100 times using random starting points with eigenvalues belonging to the interval $(0,  \; 20)$. In the second set of tests, we set $n=50$ and varied the interval that contains the eigenvalues of the starting points. Again, for each combination, the method was run 100 times using random starting points. The results for the first and second set of tests are in Table~\ref{tab:prob3}~(a) and~(b), respectively. First column of Table~\ref{tab:prob3}~(a) informs the considered dimension, while 
the first column of Table~\ref{tab:prob3}~(b) contains the interval for the eigenvalues of the starting points. Columns ``$\%$'', ``it'', and ``nfev'' are as in Table~\ref{tab:prob1}.


\begin{table}[htb!] 
{\footnotesize
 \centering 
 \begin{minipage}[b]{0.50\linewidth}
 \centering 
\begin{tabular}{|c|c|c|c|} \hline \rowcolor[gray]{0.9}
 $n$ & \% & it & nfev  \\ \hline 
10 & 100.0 &  18.2 &  110.2  \\ \hline 
20 & 100.0 &  19.9 &  140.4  \\ \hline 
50 & 100.0 &  14.2 &  114.9  \\ \hline 
100 & 100.0 &  15.2 &  138.2  \\ \hline 
150 & 100.0 &  27.1 &  271.5  \\ \hline 
\end{tabular} \\
(a)
\end{minipage}\hfill
\begin{minipage}[b]{0.50\linewidth}
\centering 
\begin{tabular}{|c|c|c|c|} \hline  \rowcolor[gray]{0.9}
 $\lambda_i(X_0)$ & \% & it & nfev  \\ \hline 
$(0 \; 10)$  & 98.0 &  14.2 &  114.6  \\ \hline 
$(0 \; 100)$  & 99.0 &  14.6 &  117.6  \\ \hline 
$(0 \; 500)$  & 99.0 &  15.0 &  121.0  \\ \hline 
$(0 \; 1000)$  & 100.0 &  15.1 &  121.6  \\ \hline 
$(0 \; 2000)$  & 100.0 &  15.2 &  122.4  \\ \hline 
\end{tabular} \\
(b)
\end{minipage}\hfill
\caption{Performance of the Riemannian gradient method in Problem~3 varying: (a) the dimension; (b) the domain of the starting points.}
\label{tab:prob3}}
 \end{table} 
 
 
 The highlight of Table~\ref{tab:prob3} is that the Riemannian gradient method was robust with respect to the dimension and to the choice of the starting points. Furthermore, except for the case where $n = 150$, it was not very sensitive to the variation of the dimension or to the domain of the starting points.
 
For comparative purposes, we implemented and tested the Euclidean method in Problem~3. For $n=5$ (respectively, $n=10$), 15 (respectively, 96) out of the 100 considered starting points resulted in an iteration history that reached the maximum number of iterations allowed. Finally, we observe that, by using~\eqref{eq:GoedSPD} and the function~\eqref{eq:fpdm4}, the Riemannian and the Euclidian gradient iteration becomes, respectively,
$$X_{k+1}=\left[ \det(X_k)^{2a}  e^{b} \right]^{-t_k}X_{k} \qquad k=0,1, \ldots,$$
and 
$$X_{k+1}=X_{k}-t_k\left[2a\ln(\det(X_k))-b\right]X_k^{-1}, \qquad k=0,1, \ldots,$$
where the steep-size $t_k>0$  is computed  according to the adopted line search strategy. Thus, we can see that the Riemannian gradient iterations are simpler and have a lower computational cost to be performed.

\vspace{0.5cm}
\noindent{\bf Problem 4.} We set $n=100$, $a=b_1=b_2=1$ and $c=0.5$ in function~\eqref{eq:fpdm}, fulfilling $c<a b_1$. We tested the Riemannian gradient method with each of the three strategies by running each combination 100 times using random starting points with eigenvalues belonging to the interval $(0, \; 20)$. The results in Table~\ref{tab:prob4} are given as in the previous tables.

\begin{table}[htb!] 
{\footnotesize
 \centering 
\begin{tabular}{|c|c|c|c|c|c|c|c|c|} \hline 
 \multicolumn{3}{|c|}{\cellcolor[gray]{.9} Strategy 1} & \multicolumn{3}{c|}{\cellcolor[gray]{.9} Strategy 2} & \multicolumn{3}{c|}{\cellcolor[gray]{.9} Strategy 3} \\ \hline \rowcolor[gray]{.9} 
 \% & it & nfev & \% & it & nfev & \% & it & nfev \\ \hline 
100.0 & 452.5 &  453.5  & 99.0 &  15.3 &   21.3 & 100.0 &  15.3 &   70.9 \\ \hline 
\end{tabular} 
\caption{Performance of the Riemannian gradient method with the different line search strategies.}
\label{tab:prob4}}
 \end{table} 

For Strategy~1, since the Lipschitz gradient constant satisfies $L<ab_1^2n$, we used the Lipschitz stepsize $t_k=1/(ab_1^2n)<1/L$, for all $k=1,2,\ldots$. Overall, as can be seen in Table~\ref{tab:prob4}, the Riemannian method with Lipschitz stepsizes is clearly the least efficient, requiring an exceedingly large number of iterations. In this case the method performs one function evaluation per iteration. The poor performance is due to the short stepsizes in all iterations. On the other hand, we point out that the efficiency of the Riemannian gradient method with Lipschitz stepsize is closely related to an accurate estimate of the Lipschitz gradient constant.
 
%

Remark~\ref{rem:strategy3} helps to explain the results of Table~\ref{tab:prob4} for Strategies~2 and 3. Regardless of the starting point, Algorithm~\ref{sec:gradient} with both strategies performed exactly the same number of iterations. Additionally, in a typical run, the stepsizes were non-increasing. Therefore, overall, by Remark~\ref{rem:strategy3}, the  adaptive scheme in Strategy~2 required fewer function evaluations per iteration then the Armijo line search of Strategy~3.

Despite the simple linesearch mechanisms employed here, the numerical results indicate that, as it has to be expected, the efficient implementation of linear search algorithms can significantly improve the Riemannian gradient method.


\subsection{The Riemannian  center  of  mass}\label{Sec:NumExpKM}
The  Riemannian center of mass and so called Karcher mean is a specific instance of a geometric mean for points in Riemannian manifolds. It has several practical applications and has appeared in many papers, we refer the reader to \cite{BiniIannazzo2013, JeurisVandebrilVandereycken2012, SraHosseini2015} and the references therein. 

\subsubsection{The  center  of  mass  on the SPD matrices cone}\label{Sec:NumExpKM}

 Denotes by  $\left\|\cdot\right\|_F$ the Frobenius norm associated to the  inner product  $\langle U,V \rangle_F\coloneqq \emph{tr} (VU)$, for all $U,V\in {\mathbb P}^n_{++}$. Let   $d$ be the Riemannian distance defined in $\mathcal{M}\coloneqq ({\mathbb P}^n_{++}, \langle \cdot , \cdot \rangle)$, i.e.,  
\begin{equation}\label{eq:KarcherMeanSDP}
d(A, X)=\left\|\ln\left(X^{-1/2}AX^{-1/2}\right)\right\|_F, \qquad A,X\in {\mathbb P}^n_{++},
\end{equation}
 see \cite{NesterovTodd2002}.  The Karcher mean   of $m$  positive definite matrices  $A_1,\ldots,A_m \in {\mathbb P}^n_{++}$ is the unique  solution of the optimization  problem 
\begin{equation} \label{eq:CMProbSDP}
\min \left\{ f(x):=\frac{1}{2}\sum_{j=1}^{m} \left\|\ln\left(X^{-1/2}A_iX^{-1/2}\right)\right\|_F^2 ~:~  X\in {\mathbb P}^n_{++}\right\}.
\end{equation}
 Indeed,  $f$ is a strong  convex  function in  $\mathcal{M}$ due to the square of the distance \eqref{eq:KarcherMeanSDP} be    strongly   convex in  $\mathcal{M}$, see for example \cite{CruzNetoFerreiraLucambio2002}.  Since $f$ is a strong  convex  function, all sub-level sets of $f$ are bounded. As a consequence, $f$  has  Riemannian  Lipschitz  gradient on each sublevel set of $f$. Finally, we remark that  \eqref{eq:KarcherMeanSDP} is not an  Euclidean  convex  function. 
 By   \cite{JeurisVandebrilVandereycken2012} and using \eqref{eq:Grad},  we conclude that
\begin{equation}\label{eq:GEMeanKarcher}
\grad f(X)=\sum_{i=1}^{m}X^{1/2}\ln\left(X^{1/2}A_{i}^{-1}X^{1/2}\right)X^{1/2}.
\end{equation}
 Thus, by using~\eqref{eq:GoedSPD}  and~\eqref{eq:GEMeanKarcher},  the Riemannian gradient iteration for  solving~\eqref{eq:CMProbSDP}  is
$$X_{k+1}=X_k^{1/2}e^{-t_k\sum_{i=1}^{n}\ln\left(X_k^{1/2}A_{i}^{-1}X_k^{1/2}\right)}X_k^{1/2},  \qquad k=0, 1, \ldots.$$
see, for example, \cite{ZhangSra2016}.

In~\cite{AfsariTronVidal2013}, Afsari {\it et al.} studied the convergence of the Riemannian gradient method with a Lipschitz stepsize for the center of mass problem in a manifold with curvature bounded from above and below. The stepsize is defined from a local estimate for the Lipschitz gradient constant. Consider problem~\eqref{eq:CMProbSDP}, and let $r>0$ be such that $A_1,\ldots,A_m \subset B(X_0,r)$, where $B(X_0,r)$ is the open ball with center $X_0$ and radius $r$. They showed that it is possible to achieve convergence with $t_k=t$ for all $k = 0,1,\ldots$,  where $t\in (0,2\bar{t})$ and
\begin{equation}\label{tbar}
\bar{t}=\frac{1}{4 r \coth(4 r)}.
\end{equation}
Recently, Bento {\it et al.}~\cite{glaydstontr} extended the convergence of the gradient method to the Hadamard setting for continuously differentiable functions which satisfy the Kurdyka-Lojasiewicz inequality. In particular, they proposed a Riemannian gradient method with Armijo line search for problem~\eqref{eq:CMProbSDP}. Basically, their proposal coincides with Algorithm~\ref{sec:gradient} with Strategy~3. 

We tested Algorithm~\ref{sec:gradient} with each strategy on a set of 200 randomly generated problems~\eqref{eq:CMProbSDP} with $n=200$ and $m=5$, $10$, $20$ or $50$. For each value of $m$ we considered 50 problem instances. Let us clarify how a matrix $A$ was defined. First, we randomly generated an orthonormal matrix $U$ and a diagonal matrix $D$ with elements belonging to $(0,100)$. Then, we set $A=UDU^T$ ensuring that $A\in {\mathbb P}^n_{++}$. Given a problem instance with data $A_1,\ldots,A_m \in {\mathbb P}^n_{++}$, we defined the starting point $X_0$ as the {\it explog} geometric mean given by
$$X_0 \coloneqq \exp\left( \frac{1}{m}\sum_{i=1}^m \ln(A_i)\right),$$
see, for example, \cite{ANDO2004305}. For Strategy~1 the Lipschitz stepsize $t$ was defined according to~\cite{AfsariTronVidal2013}. We set $t=1.99\bar{t}$, where $\bar{t}$ is given by~\eqref{tbar}. Radius $r$ can be calculated by computing the maximum distance of $X_0$ to each matrix $A_i$, $i=1,\ldots,m$. Numerical comparisons are reported in Figure~\ref{fig:prob6} using performance profiles~\cite{pprofile}. We adopted the number of functions evaluations and CPU time as performance measurements. 

 \begin{figure}[h!]
\begin{minipage}[b]{0.50\linewidth}
\begin{figure}[H]
	\centering
		\includegraphics[scale=0.5]{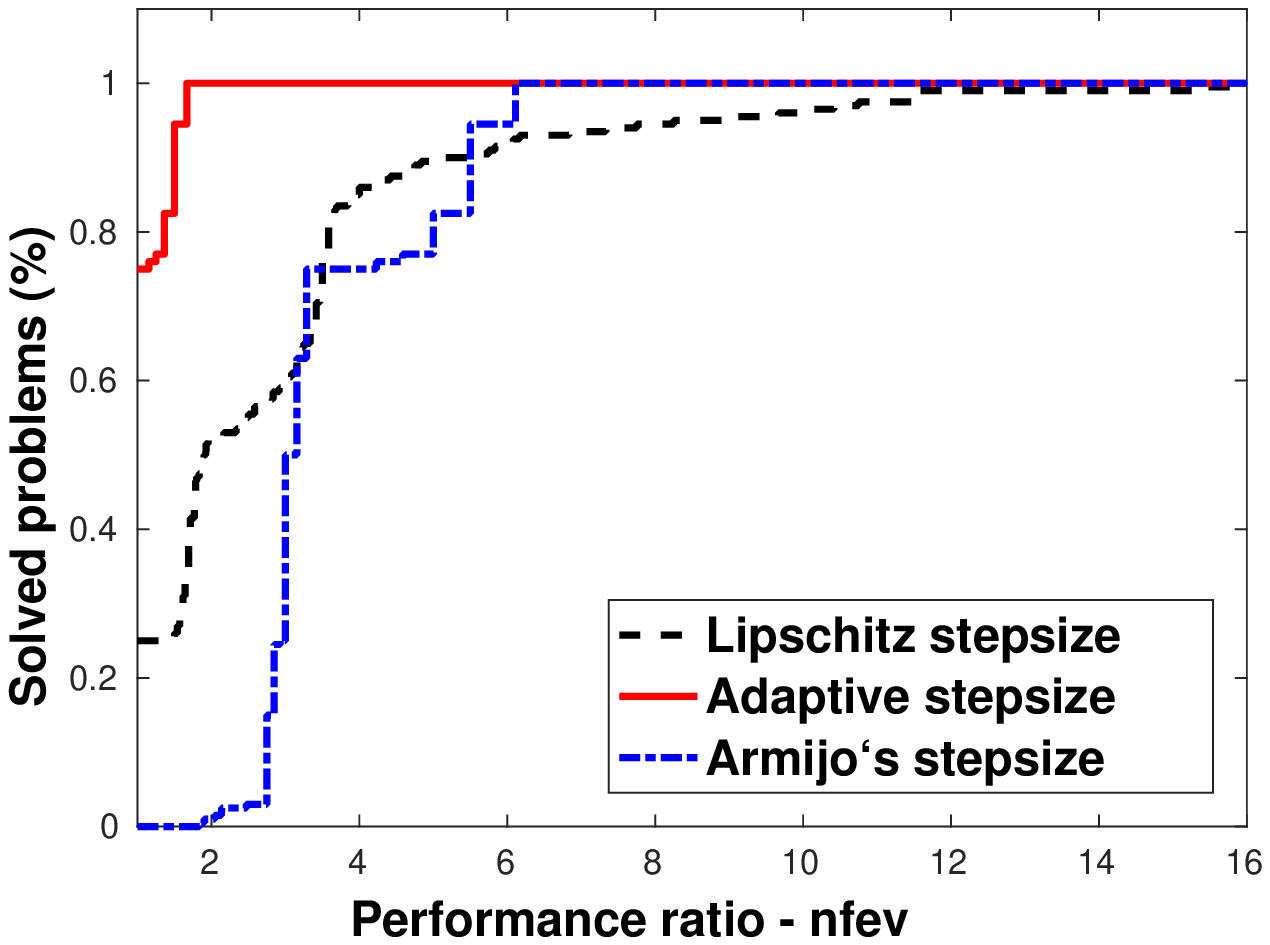}\\
	\footnotesize	(a)
\end{figure}

\end{minipage} \hfill
\begin{minipage}[b]{0.50\linewidth}

\begin{figure}[H]
	\centering
		\includegraphics[scale=0.5]{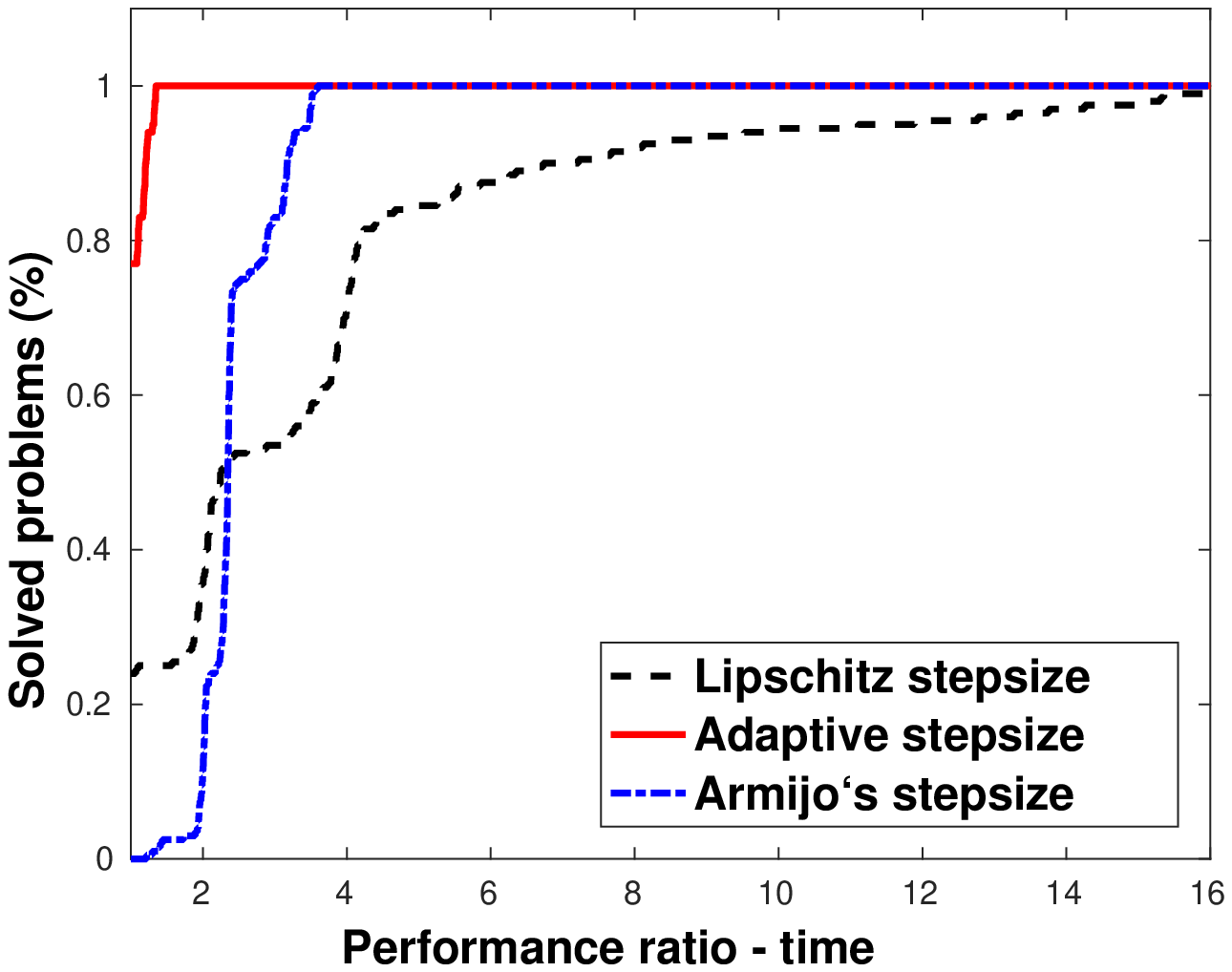}\\
	\footnotesize	(b)
\end{figure}
\end{minipage}\hfill
\caption{Performance profile comparing the Riemannian gradient method with different line search strategies using as performance measurement: (a) number of function evaluations; (b) CPU time.}
\label{fig:prob6}
\end{figure}

As can be seen, Algorithm~\ref{sec:gradient} with Strategy~2 is the most efficient on the chosen set of test problems. Efficiencies of the methods are $25.0\%$ (respectively, $24.0\%$), $75.0\%$ (respectively, $76.0\%$), and $0.0\%$ (respectively, $0.0\%$) respectively, considering the number of function evaluations (respectively, CPU time) as performance measurement. Efficiency of Algorithm~\ref{sec:gradient} with Strategy~3 is $0.0\%$ because Strategy~2 outperformed Strategy~3 in all considered instances. Curiously, $m = 20$ in all problems for which  Strategy~1 was the most efficient. Robustness are $99.5 \%$, $100.0 \%$, and $100.0 \%$ respectively, see Table~\ref{tab:prob6}. Only in a problem instance Algorithm~\ref{sec:gradient} with Strategy~1 reached the maximum number of iterations allowed.

 \begin{table}[htb!] 
  \footnotesize
  \centering
\begin{tabular}{|c|c|c|} \hhline{~*2{-}|}  
\multicolumn{1}{c|}{} & \multicolumn{1}{c|}{\cellcolor[gray]{0.9} Efficiency (nfev -- CPU time) ($\%$) } & \multicolumn{1}{c|}{\cellcolor[gray]{0.9} Robustness ($\%$)} \\ \hline
 Strategy~1   & 25.0  -- 24.0  & 99.5 \\ \hline
 Strategy~2   & 75.0 -- 76.0   & 100.0\\ \hline
 Strategy~3   & 0.0  -- 0.0    & 100.0\\ \hline
\end{tabular}
\label{tab:prob6}
\caption{}
\end{table}

The similarity of the Figures~\ref{fig:prob6}~(a) and~(b) suggests that the number of function evaluations is a good indicator of performance. Indeed, evaluating function $f$ is computationally expensive, since it involves inverting $X$ and computing $m$ matrix logarithms. This implies that linear search schemes must be carefully formulated for the center of mass problem. Overall, the naive implementation of the Armijo line search in Strategy~3 was overcome by the method with Lipschitz stepsize. On the other hand, the results indicate that the adaptive search proposed in Strategy~2 is a promising scheme worth to consider.

\subsubsection{The  center  of  mass  on the positive orthant}\label{Sec:NumExpKM}

 Let    ${\mathcal M}\coloneqq(\mathbb{R}_{++}^n,G)$  be  the Riemannian manifolds defined in Section~\ref{Sec:ExamplesPO} and $d$  the Riemannian distance  associated.  Hence,  we have 
\begin{equation}\label{eq:KarcherMean1}
d^2(y, x)= \sum_{i=1}^n \ln^2\left(\frac{y_i}{x_i}\right), \qquad   y=(y_1, \ldots, y_n), ~ x=(x_1, \ldots, x_n) \in \mathbb{R}_{++}^n.
\end{equation}
The center  of  mass   of $m$ points  $w^1,\ldots,w^m \in {\mathbb R}^n_{++}$ is the unique  solution of the optimization  problem 
\begin{equation} \label{eq:CMProb}
\min \left\{ f(x):=\frac{1}{2}\sum_{j=1}^{m}  d^2(w^j, x) ~:~   x\in \mathbb{R}_{++}^n\right\}.
\end{equation}
Since the square of the distance~\eqref{eq:KarcherMean1} is strongly   convex in  $\mathcal{M}$, then $f$ is a strong  convex  function in  $\mathcal{M}$, see for example \cite{CruzNetoFerreiraLucambio2002}.   
By using  \eqref{eq:GradExPO},   we conclude that  
$$\grad f(x)=\left(x_1 \sum_{j=1}^{m}\ln\left(\frac{x_1}{w^j_{1}}\right),\ldots,x_n \sum_{j=1}^{m}\ln\left(\frac{x_n}{w^j_{n}}\right)\right),$$
where $ x=(x_1, \ldots, x_n) \in \mathbb{R}_{++}^n$. Problem~\eqref{eq:CMProb} has closed solution $x^*=(x_1^*,\ldots,x_n^*) \in \mathbb{R}_{++}^n$ given by
$$x_i^* = \left(\prod_{j=1}^{m}   w_i^j \right)^{\frac{1}{m}},$$
for all $i=1,\ldots,m$. Indeed, direct calculations show that $\grad f(x^*)=0$. 

Due to the closed-form solution, we use problem~\eqref{eq:CMProb} to illustrate the results on iteration-complexity bound of Section~\ref{Sec:IteCompAnalysis}. We consider the Riemannian gradient algorithm with Lipschitz stepsize. Note that the set of positive definite diagonal  matrices  can be identified  with $\mathbb{R}_{++}^n$. Thus, problem~\eqref{eq:CMProb} can be seen as a particular case of problem~\eqref{eq:CMProbSDP} for positive definite diagonal matrices. Given $w^1,\ldots,w^m \in {\mathbb R}^n_{++}$ and defining $A_i=\mbox{diag}(w^i)$ for all $i=1,\ldots,m$, we defined the Lipschitz stepsize as in Section~\ref{Sec:NumExpKM}.

We set $n=100$, $m=5$ and randomly generated the elements of $w^1,\ldots,w^m$ and initial point $x_0$ from a uniform distribution on $(0, 100)$. The computed Lipschitz stepsize was set to $t \approx 0.06$. The Riemannian gradient algorithm stopped declaring ``solution was found'' with 30 iterations. Figures~\ref{fig:prob5}~(a) and~(b) report the function values of the left and right hand sides of inequalities~\eqref{eq:complexity} and~\eqref{eq:complexity2}, respectively.

 \begin{figure}[h!]
\begin{minipage}[b]{0.50\linewidth}
\begin{figure}[H]
	\centering
		\includegraphics[scale=0.50]{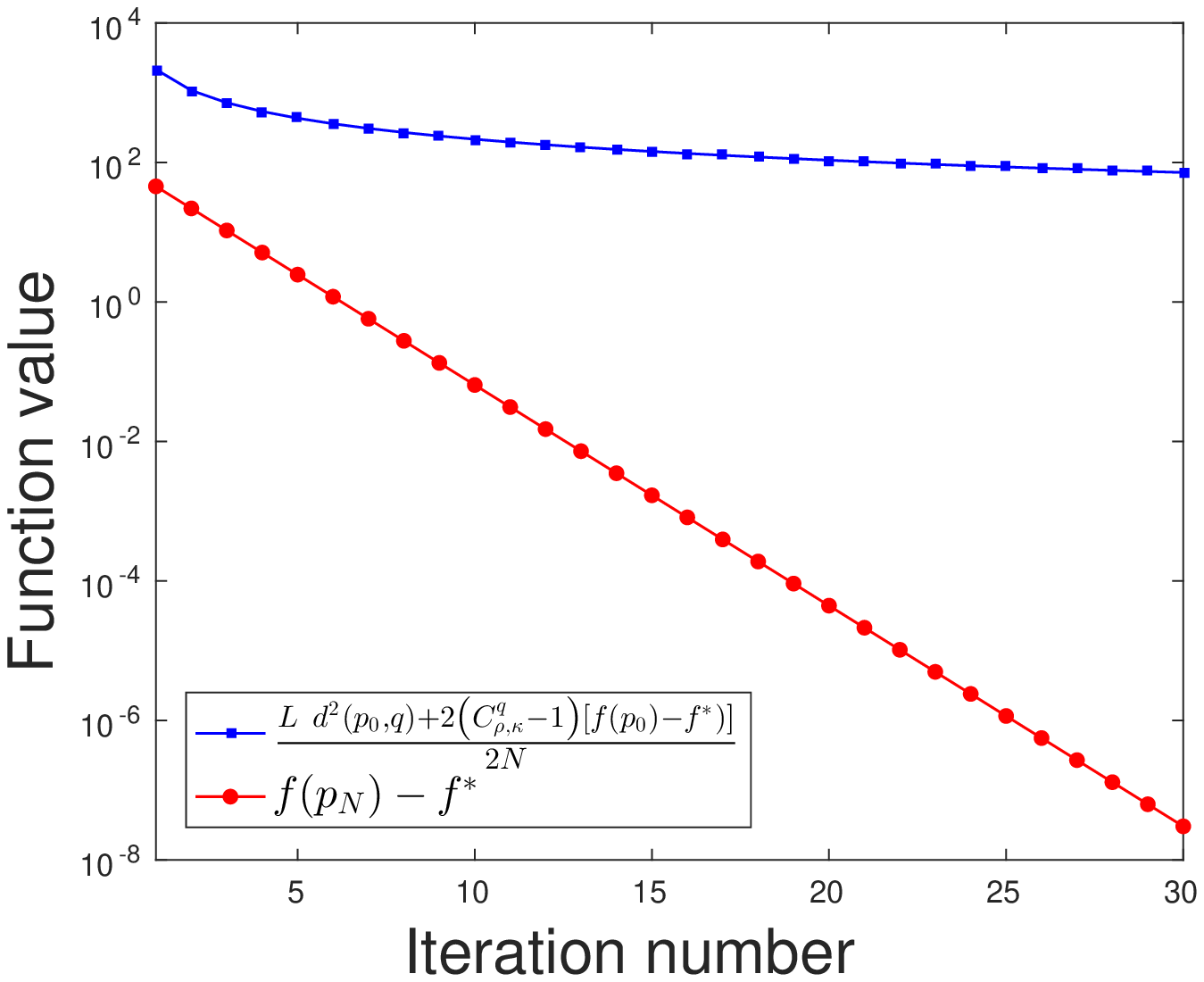}\\
	\footnotesize	(a)
\end{figure}

\end{minipage} \hfill
\begin{minipage}[b]{0.50\linewidth}

\begin{figure}[H]
	\centering
		\includegraphics[scale=0.50]{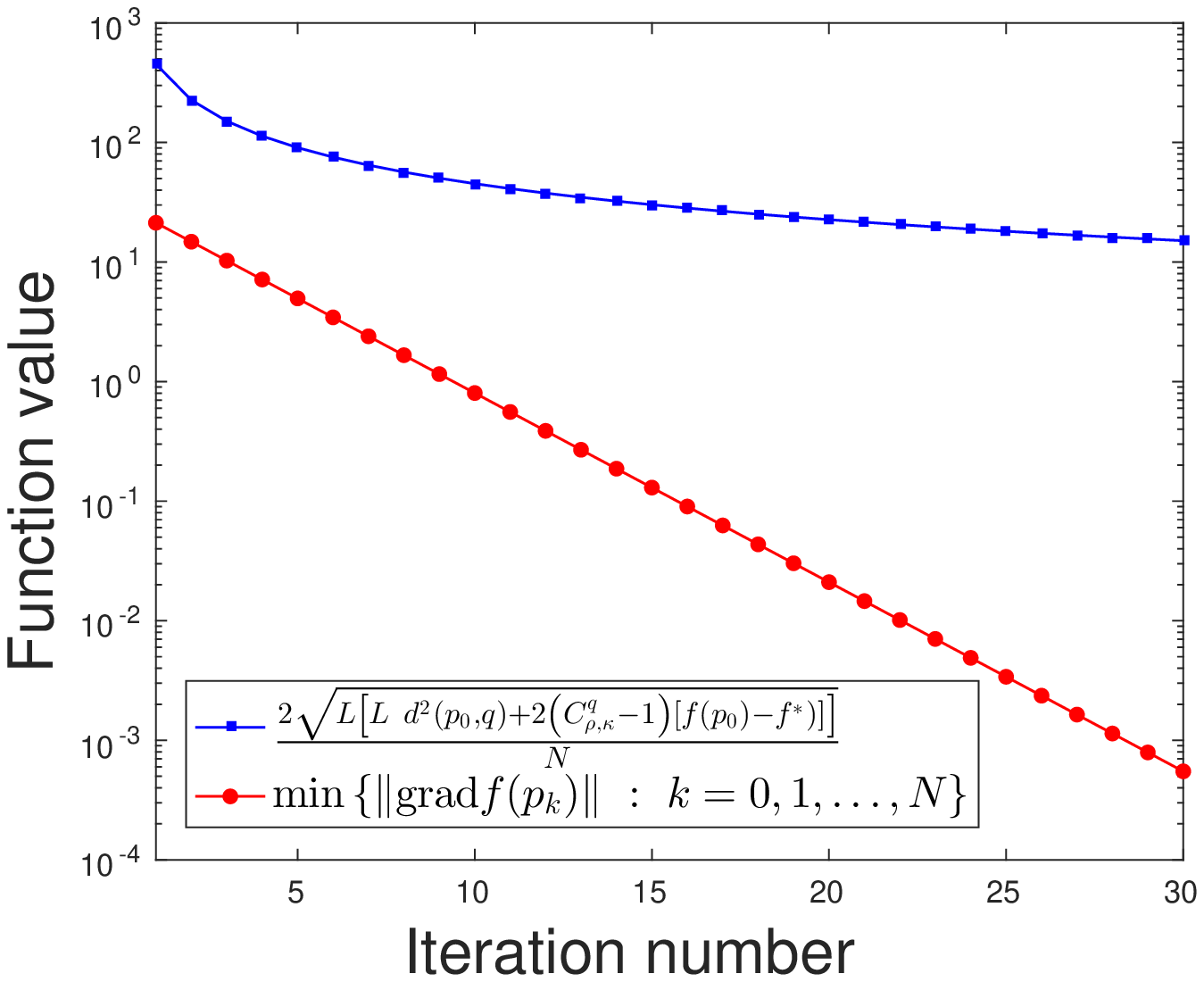}\\
	\footnotesize	(b)
\end{figure}
\end{minipage}\hfill
\caption{Iteration-complexity bound for the Riemannian gradient method with Lipschitz stepsize related to: (a) objective function value -- Theorem~\ref{th:grad2} ; (b) norm of the Riemannian gradient -- Corollary~\ref{cr:icgm}.}
\label{fig:prob5}
\end{figure}

As can be seen in Figure~\ref{fig:prob5}, the  iteration-complexity  bounds related to the objective function value and the norm of the Riemannian gradient are always respected, see Theorem~\ref{th:grad2} and Corollary~\ref{cr:icgm}. This illustrate the practical reliability of our iteration-complexity results.

\section{Conclusions}\label{Sec: Conclu}
In this paper,  the behavior of  the  gradient  method  for convex  optimization problems on Riemannian manifolds  with lower bounded sectional curvature were analyzed. We considered   three different finite procedures for determining the stepsize, namely, constant stepsize,  adaptive procedure  and  Armijo's procedure. As far as we know,  the full convergence of the sequence generated by this   method with these  three strategies  is a new contribution of this paper, which adds important results in the available convergence theory.  Besides, under mild assumptions,  we showed that the  iteration-complexity bound related to the  method  is $\mathcal{O} \left(1/\epsilon\right)$ for finding a point $p_N\in \mathcal{M}$ such that $ f(p_N)-f^*< \epsilon$.   The numerical experiments  provided  illustrate the effectiveness of the method in this new setting  and certify the conclusions suggested by the theoretical  results.  Despite the simple linesearch mechanisms employed here, the numerical results indicate that, as it has to be expected, the efficient implementation of linear search algorithms can significantly improve the Riemannian gradient method. In particular, the effectiveness of the method to find the Riemannian mass center and the so-called Karcher's mean is presented, indicating that the adaptive procedure is a promising scheme that is worth considering. We expect that this paper will contribute to the development of  studies of optimization methods in the Riemannian setting.  Finally, it would be interesting to analyze stochastic versions of the the gradient method by using adaptive procedures.  
\def\cprime{$'$}

\end{document}